\documentclass{glasnik}
\usepackage{glasnik}
\usepackage{graphicx}
\usepackage[T1]{fontenc}
\usepackage{amsfonts,amssymb,amsmath}
\usepackage{graphicx} 
\usepackage[slovene,english]{babel}
\usepackage{graphpap,comment}   
\newcommand{\hide}[1]{}
\def\:{\mkern1mu\colon}

\def\RR{\mathbb{R}}
\def\PP{\mathbb{P}}

\def\pp{\mathbf{p}}
\def\qq{\mathbf{q}}

\def\nn{\mathbf{n}}
\def\pxv{P_{\langle\mathbf{X},S\rangle}}
\def\pyv{P_{\langle\mathbf{Y},S\rangle}}
\def\pxm{P_{\langle\mathbf{x},s\rangle}}
\def\pym{P_{\langle\mathbf{y},s\rangle}}

\DeclareMathOperator{\psin}{psin}
\DeclareMathOperator{\vol}{vol}

\newtheorem{lema}[theorem]{Lemma}
\newtheorem{coro}[theorem]{Corollary}

\begin{document}

\title{GEOMETRIC CONSTRUCTIONS ON CYCLES IN $\RR^n$}
\UniCountry{University of Ljubljana, Slovenia}
\author[B. Jur\v ci\v c Zlobec and N. Mramor Kosta]{Borut Jur\v ci\v c Zlobec and Ne\v za Mramor Kosta}
\address[Borut Jur\v ci\v c Zlobec]{Department of Electrical Engineering, University of Ljubljana\\ 1000 Ljubljana\\ Slovenia}
\email[Borut Jur\v ci\v c Zlobec]{borut.jurcic@fe.uni-lj.si}
\address[Ne\v za Mramor Kosta]{Department of Computer and Information Science, University of Ljubljana\\ and Institute of Mathematics, Physics and Mechanics\\1000 Ljubljana\\ Slovenia}
\email[Ne\v za Mramor Kosta]{neza.mramor@fri.uni-lj.si}



\abstract{
  In Lie sphere geometry, a cycle in $\RR^n$ is either a point or an oriented
  sphere or plane of codimension $1$, and it is represented by
  a point on a projective surface $\Omega\subset \PP^{n+2}$. The Lie product, a bilinear form on 
  the space of homogeneous coordinates $\RR^{n+3}$, provides an algebraic description of geometric properties of 
  cycles and their mutual position in $\RR^n$. In this paper we discuss 
  geometric objects which correspond to the intersection of
  $\Omega$ with projective subspaces of $\PP^{n+2}$. Examples of such
  objects are spheres and planes of codimension $2$ or more, cones and tori. 
  The algebraic framework which Lie geometry provides gives rise to simple and efficient
   computation of invariants of these objects, their properties and their mutual position in $\RR^n$. 
}
\maketitle
\noindent

\section{Introduction}
In his dissertation \cite{Lie} published  in 1872, Sophus Lie introduced his Lie 
geometry of oriented spheres which is based on a bijective 
correspondence between {\em oriented geometric cycles}, that is, planes and spheres of codimension $1$, 
and points on a quadric surface $\Omega$ in the projective space 
$\PP^{n+2}$.  Geometric relations like tangency,
angle of intersection, power, etc. are expressed in terms of the Lie product, a nondegenerate bilinear form on the  the space $\RR^{n+3}$
of homogeneous coordinate vectors. 
Lie geometry is an extension of the perhaps better known M\" obius geometry 
of unoriented cycles in $\RR^n$. Both Lie and M\" obius geometries provide 
an algebraic framework for computing geometric invariants of cycles and expressing their mutual position. This makes Lie
geometry an appropriate language for dealing with geometric
constructions on spheres and planes in $\RR^n$. It has been used to 
study a variety of geometric problems on circles and lines in the
plane and design algorithms for finding their solutions, for example in
\cite{Y}, \cite{R}, \cite{FS}. In \cite{BJZNMK1}, Lie geometry was used 
to analyze the existence and properties of solutions of geometric constructions 
associated to the Apollonius construction in $\RR^n$. In \cite{Kn} 
the special case of the Apollonius problem in the plane was considered. In
\cite{BJZNMK2}, simple algorithms for symbolic solutions of a number
of such geometric constructions were given. 
A thorough treatment of Lie geometry can be found in
\cite{Behnke} or \cite{Cecil}. 

In this paper we study geometric objects in $\RR^n$ which are obtained as 
intersections of projective subspaces of $\PP^{n+2}$ with the quadric $\Omega$. 
For example, the intersection of $\Omega$ with a $(k+1)$-dimensional 
projective subspace of $\PP^{n+2}$ spanned by $k$ cycles and
a special element 
$r\in\PP^{n+2}$ determines a $k$-parametric family of geometric cycles in 
$\RR^n$. If the spanning cycles correspond to intersecting geometric objects, this determines a {\em subcycle}, that is, a sphere 
or plane of codimension $k$ in $\RR^n$ which is the common intersection 
of all geometric cycles belonging to the family. Such a family is known
as a {\em Steiner family}. Similarly, the intersection of $\Omega$ with a 
$(k+1)$-dimensional projective subspace spanned by $k$ cycles and a second special 
element $w\in\PP^{n+2}$ determines a $k$-parametric {\em cone family}. If the spanning cycles have a common
tangent plane, this generates a cone in $\RR^n$ consisting of 
all points of tangency of the cycles of the family to the common tangent planes. 

Many geometric 
properties of these objects can be computed from simple, 
easily computable algebraic invariants of the corresponding
projective subspaces. Such an invariant is for example the Lie form restricted 
to the linear subspace of homogeneous coordinate vectors. First of 
all, its sign determines which projective subspaces have nonempty intersections 
with $\Omega$ and thus define geometric objects. Second, quotients of 
determinants, the so-called {\em discriminants}, are algebraic invariants of the obtained geometric objects
which reflect their geometric properties. Certain projective transformations (that is, linear transformations on the homogeneous coordinates), in particular Lie projections and Lie reflections, enable a simple, easily implementable computation of geometric characteristics.  

Following is a description of the main results of this paper. In section \ref{cikli} we give a short summary 
of Lie geometry of oriented geometric cycles where we refer the reader to \cite{BJZNMK1} or
(with minor changes in notation) to \cite{Cecil} for the missing proofs and details.  In section \ref{druzine} we
introduce general $s$-families, that is, families of cycles which are obtained as intersections of $\Omega$ with a projective
subspace of $\PP^{n+2}$ containing a distinguished cycle $s$. We then describe our main algebraic tools: determinants, Lie projections and Lie reflections. 
We focus on {\em hyperbolic} families, where the determinant of the corresponding subspace is 
negative, since these determine geometric objects of interest to us. Depending on a further cycle $s'$ we define the $s'$-discriminant of a hyperbolic $s$-family. 
The Lie projection $c$ of $s'$ onto the corresponding projective subspace has a special role in the family: it generates the cycle of the family 
with the minimal $s'$-discriminant. 
We show that the $w$-discriminant of a hyperbolic $r$-family gives the radius of the corresponding subcycle, the projection $c$ determines the cycle in the family 
with the minimal radius, and the planes of the family correspond to elements in the projection of the dual subspace $\ell=s'^\perp$. 
In the case of hyperbolic $w$-families, the $r$-discriminant gives the angle at the vertex of the cone, and $c$ determines the plane, 
orthogonal to the axis of the cone while its dual subspace determines the points of the family, in particular, in the case of a $1$-parametric cone family, the vertex of the cone.

In section \ref{mesan} we consider the mutual position of an
$s$-family and a cycle. We define the $s$-discriminant of the two
objects and show that its value coincides with the extreme value of the $s$-discriminant on pairs consisting of the given cycle and any cycle from the family and 
that it is achieved on the Lie projection of the given cycle onto the family. We prove that the $r$-discriminant of a cycle and a subcycle corresponds to the minimal 
angle of intersection (if it exists) or the maximal angle under which the given cycle is seen from the cycles of the family (if the intersection does not exist), 
and that this extreme value is achieved on the projection 
of the given cycle onto the family. On the other hand, the discriminant also gives the angle of the segment connecting the center of the given cycle and a point on the subcycle above the plane in $\RR^n$ in which the subcycle lies. The $w$-discriminant of a cycle and a cone corresponds to the tangential distance of the given cycle to the cone, which coincides with the minimal tangential distance between cycles of the cone family and the given cycle and is again achieved on the Lie projection of the cycle onto the family. 
 
In section \ref{dvaenaka} we generalize this to the case of two
families of the same dimension and define their $s$-discriminant.  We show that the value of the discriminant coincides with the extreme value of the $s$-discriminant on pairs consisting of a cycle from each family, and that it is achieved at the fixed points of a product of two Lie projections. Finally, we illustrate these results on the case of two subcycles and two cones in $\RR^3$: the sign of the $r$-discriminant of two subcycles determines whether they are linked or unlinked, and the fixed point pairs determine the cycles in the two $r$-families with the largest and the smallest angle, while the sign of the $w$-discriminant of a pair of cones determines whether their tangential distance exists and the fixed point pairs correspond to pairs of cycles from each $w$-family with minimal tangential distance. 

All constructions which appear in this paper as well as in \cite{BJZNMK1} and
\cite{BJZNMK2} have been implemented in a Mathematica package which can 
be found at {\tt http:matematika.fe.uni-lj.si/people/borut/Lie/}. The
examples in this paper have all been generated with this package.

\section{Cycles}\label{cikli}

Throughout this paper we will use the following convention: a lowercase letter will denote an element of a projective space, and the corresponding upper case letter its homogeneous coordinate vector.

We start with an informal geometric description of the 
M\" obius and Lie coordinates of cycles in $\RR^n$. 
In M\" obius geometry, unoriented cycles in $\RR^n$ are 
represented as points on or outside a quadric surface $Q$ in the projective space $\PP^{n+1}$ in the following way. The vector space $\RR^{n+2}$ of homogeneous coordinate vectors of points in $\PP^{n+1}$ is given the {\em Lorentz metric} or 
{\em M\" obius product},  an indefinite bilinear form of index 
$1$  which we denote by $(\, ,\, )$.  The quadric $Q$ consists of points $z\in \PP^{n+1}$ such that $(Z,Z)=0$. The quadric $Q$ divides the points of $\PP^{n+1}$ into three types: a point $z$ on $Q$ is {\em lightlike}, a point 
$z$ such that $(Z,Z)>0$ is {\em spacelike}, and a point $z$ such that $(Z,Z)<0$ is {\em timelike}. The cycles of $\RR^n$ are represented by lightlike and spacelike 
points of $\PP^{n+1}$ by identifying $\RR^n$ with $\{(\zeta_1,\ldots,\zeta_{n+1})\mid \zeta_{n+1}=0\}\subset\RR^{n+1}\subset\PP^{n+1}$ and $Q$ with the paraboloid $\zeta_{n+1}=-\frac{1}{2}\sum_{i=1}^n\zeta_i^2$ 
tangent to $\RR^n$ at the origin. 
Each geometric cycle $c$ in $\RR^n$ can be obtained by intersecting $Q$ with an $n$-plane $l \subset \RR^{n+1}$ and projecting the intersection to $\RR^n$. If $c$ is a point, that is, a sphere with radius $0$, then $l$ is tangent to $Q$ and $z$ is the point of tangency. If $c$ is a sphere or plane, then $z$ is the {\em polar point} of $l$, that is,  the common point of all tangent planes to $Q$ at the points of intersection $Q\cap l$, and $z$ is spacelike. If $c$ is a sphere then $z\in\RR^{n+1}\subset \PP^{n+1}$ and if $c$ is a plane then $z\in\PP^{n+1}\setminus\RR^{n+1}$.
Figure \ref{kvadrikal} shows  this construction in the case of a $0$-dimensional cycle in $\RR^1$: a $0$-sphere in $\RR^1$ consisting of two points, denoted by $a'$ and $b'$, is obtained by projecting the intersection $\{a,b\}$ of the quadric with a line in $\RR^2$, and $z\in\RR^{2}\subset \PP^2$ is the polar point of this line.

In Lie geometry, an additional dimension is added, that is, $\PP^{n+1}$ is embedded as a projective subspace into $\PP^{n+2}$. 
The M\" obius product is extended to the {\em Lie product}, an indefinite bilinear form of index $2$ on the space of homogeneous coordinate vectors $\RR^{n+3}$. The Lie product of vectors $X,Y\in \RR^{n+3}$ will be denoted by $(X\mid Y)$. Each oriented geometric cycle in $\RR^n$ is represented by a point 
on the Lie quadric $\Omega = \{x\in \PP^{n+2} \mid (X\mid X)=0\}$. The quadric $\Omega$ is a branched double cover over the spacelike and lightlike points of $\PP^{n+1}$ intersecting $\PP^{n+1}$ in the branching locus $Q$. A point of $\RR^n$ is represented by a single point $z\in Q\subset\Omega$, while a sphere or a plane is represented by  two 
points $z^+$ and $z^-$ on $\Omega$ lying above and below $z$, one for each orientation. 
Figure \ref{kvadrikar} shows the projection from 
the Lie quadric to the M\" obius space.

\begin{figure}\label{kvadrikal}
\centering
\includegraphics[scale=0.85]{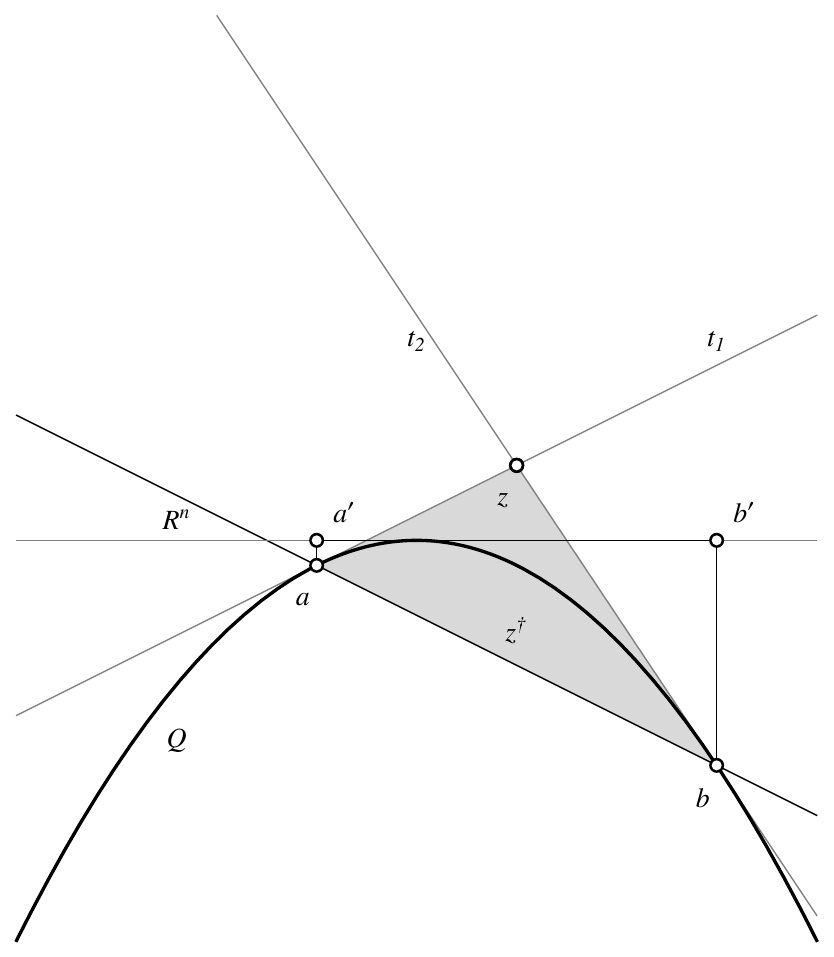}
\caption{In M\" obius geometry the $0$-dimensional
unoriented sphere in $\RR^1$ consisting of the points $a'$ and $b'$ is represented by the point $z$ in the plane 
$\RR^2$.  
The  $0$-sphere $\{a',b'\}$ is the projection onto $\RR^1$ of the points $\{a,b\}$ on the 
M\" obius  quadric $Q$ and $z$ is the polar point to the line through the points $a$ and $b$.
} 
\end{figure}

\begin{figure}\label{kvadrikar}
\centering
\includegraphics[scale=0.6]{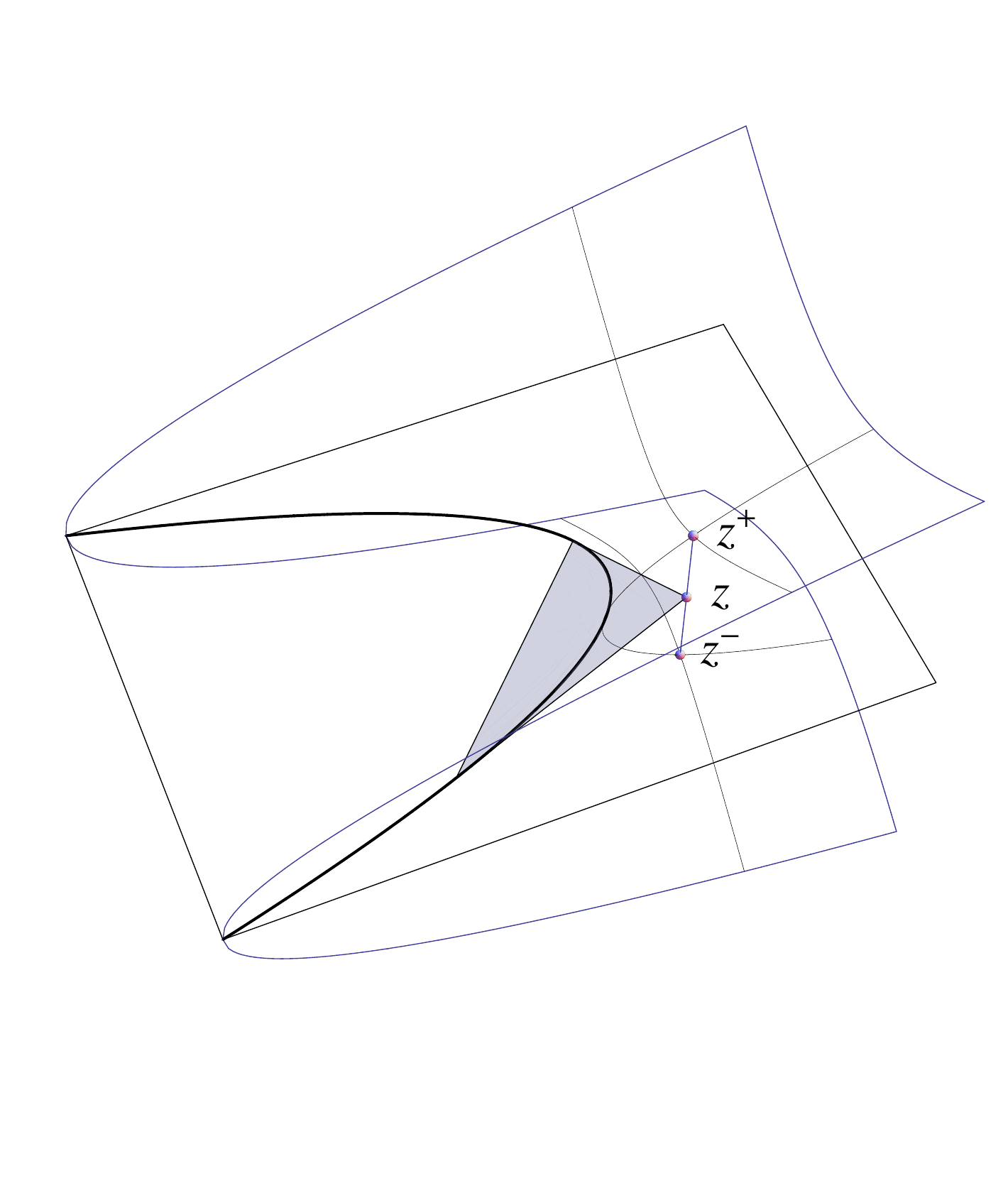}
\caption{In Lie geometry the points on and outside the M\" obius quadric in $\PP^{n+1}$ are covered by the Lie quadric $\Omega\subset \PP^{n+2}$. A spacelike point $z\in \PP^{n+1}$ which represents an unoriented sphere in M\" obius geometry is covered by the two points $z^-,z^+\in\PP^{n+2}$ representing the sphere with both orientations.}
\end{figure}

In M\" obius geometry $(X,Y)=0$ if and only if the two geometric cycles 
represented by $x$ and $y$ intersect orthogonally. 
More precisely, for any two spacelike points 
$x,y\in \PP^{n+1}$ 
the M\" obius product in suitable homogeneous coordinates gives the angle of intersection of the two cycles they represent. 
In Lie geometry the equation $(X\mid Y)=0$ describes oriented contact, that is, 
$(X\mid Y)=0$ if and only if the two oriented geometric cycles $x,y$ are tangent 
with compatible
orientations in the points of tangency. Because of this, Lie geometry is particularly suited 
for dealing with configurations of geometric cycles with certain tangency requirements, since such 
requirements are encoded by the simple linear equation $(X\mid Y)=0$. 

In order to give a precise description of the correspondence between oriented 
geometric cycles and points on $\Omega$ which can be used for specific algorithms and 
computations, it is necessary to introduce local coordinate charts on 
$\PP^{n+2}$.

We will call an
element $x\in \PP^{n+2}$ (denoted by a lowercase letter) an {\em
  algebraic cycle} (or, mostly, just a cycle). An algebraic cycle $x$
is given by a nonzero vector of homogeneous coordinates
$X \in \RR^{n+3}$ which is determined up to a scalar factor and which we denote by the
corresponding uppercase letter. 

The {\em Lie product} on $\RR^{n+3}$ is a nondegenerate bilinear form 
of index $2$ given by
$$
	(X\mid Y)=X^T\,\mathbf{A}\,Y,\qquad 
	\mathbf{A}=\begin{bmatrix}
		0 & \mathbf{0}&1& 0 \\
		\mathbf{0}& \mathbf{I}^{n}& \mathbf{0}&\mathbf{0} \\
		1 & \mathbf{0} & 0 & 0\\
		0 & \mathbf{0} & 0&-1 
	\end{bmatrix}
$$
where $\mathbf{I}^n$ denotes the $n\times n$ identity matrix. 
In coordinates, let $X=(\xi_0,\boldsymbol{\xi}_1,\xi_{2},\xi_3)$ and 
$Y=(\eta_0,\boldsymbol{\eta}_1,\eta_2,\eta_3)$ where 
$\xi_i,\eta_i\in\RR$ for $i=0,2,3$ and 
$\boldsymbol{\xi}_1,\boldsymbol{\eta}_1\in \RR^n$, then
 $$(X\mid Y)=\xi_0\eta_{2} + \boldsymbol{\xi}_1\cdot
\boldsymbol{\eta}_1+\xi_2\eta_0-\xi_3\eta_3.$$
 The vectors $X$ such that $(X\mid X)=0$ form the {\em Lie
  quadric}
$$\Omega:= \{x\in \PP^{n+2}\mid (X\mid X) = 0\}\subset \PP^{n+2}.$$
Cycles $x\in\Omega$ will be called {\em proper cycles}, while cycles
$x\notin\Omega$ will be called {\em non proper cycles}.

If $\mathbf{X}=(X_1,\ldots,X_k)$ denotes a list of homogeneous coordinate vectors, the symbol
$\langle\mathbf{X}\rangle=\langle X_1,\ldots,X_k\rangle$ will stand for the linear subspace
spanned by the vectors $(X_1,\ldots,X_k)\in \RR^{n+3}$, and the symbol
$\langle\mathbf{X}\rangle^\perp=\langle X_1,\ldots,X_k\rangle^\perp$ for the orthogonal\label{orthogonal_complement} 
complement to
$\langle\mathbf{X}\rangle$ with respect to the Lie product, i.e.
$$\langle \mathbf{X}\rangle^\perp = \{ Y\mid (X_i\mid Y)=0, i=1,\ldots,k\}.$$
Following our convention on uppercase and lowercase letters, 
$\langle\mathbf{x}\rangle=\langle x_1,\ldots,x_k\rangle$ and $\langle\mathbf{x}\rangle^\perp=\left< x_1,\ldots,x_k\right>^\perp$ will denote the projective subspace spanned by $\mathbf{x}$ and its dual 
projective subspace, respectively. 
For any nonzero vector $S\in \RR^{n+3}$, the open set $\mathcal{U}_s =
\PP^{n+2}\setminus \left<s\right>^\perp$ together with the map
\begin{equation}\label{map}
\varphi_S\:\mathcal{U}_s \to \RR^{n+3}, \quad
\varphi_S(x) := \frac{1}{(X\mid S)} \, X,
\end{equation}
where $X$ is any vector of homogeneous coordinates of $x$, 
with image in $\{X\mid (X\mid S)=1\}\cong \RR^{n+2}$, 
is a chart on $\PP^{n+2}$ specifying  local coordinates in 
$\mathcal{U}_s$. The collection
$$\{(\mathcal{U}_s,\varphi_S)\mid S\neq 0\in\RR^{n+3}\}$$
gives the standard manifold structure on $\PP^{n+2}$.  

Two
cycles and their corresponding charts have a special role in Lie
geometry: the non proper cycle $r$ with homogeneous coordinates
$R=(0,\mathbf{0},0,1)$ and the proper cycle $w$ with homogeneous coordinates
$W=(1,\mathbf{0},0,0)$. The reason for this is that the equation 
$(X\mid X)=0$ implies that either $\xi_2\neq 0$ and $X\in\mathcal{U}_w$, or $\xi_3\neq 0$ and $X\in\mathcal{U}_r$, 
so 
$$\Omega\subset \mathcal{U}_w\cup\mathcal{U}_r.$$

An oriented geometric cycle in $\RR^n$ is represented by a proper
algebraic cycle $x\in \Omega \subset \PP^{n+2}$ in the following way.
\begin{itemize}
\item The positively oriented (i.e. inward normal) and negatively
oriented sphere with center  $\pp$ and  radius $\rho$ are
represented by the cycles $x$ and $x'$ in $\mathcal{U}_w\cap
\mathcal{U}_r$ with local coordinates
\begin{equation}\label{lk}
\varphi_W(x) = \left(\nu,\pp,1,\rho\right), 
\quad \varphi_R(x) = 
\left(-\frac{\nu}{\rho},-\frac{\pp}{\rho},-\frac{1}{\rho},-1\right)
\end{equation}
\begin{equation}\label{lk1}\varphi_W(x') = \left(\nu,\pp,1,-\rho\right),\quad 
\varphi_R(x') = \left(\frac{\nu}{\rho},\frac{\pp}{\rho},\frac{1}{\rho},-1
\right),\end{equation}
respectively, where $\nu=\frac{1}{2}(\rho^2-\|\pp\|^2)$ and $\rho>0$.

\item A point $\pp\in \RR^n$, i.e. a sphere with radius $0$, is represented 
by the cycle $x\in \langle r\rangle^\perp\subset \mathcal{U}_w$ with local coordinates 
$\varphi_W(x)= (\nu,\pp,1,0)$, where $\nu=-\|p\|^2/2$.

\item A plane with normal $\nn$, where $\|\nn\|=1$, containing the point 
$\qq$ is represented by the cycle $x\in \langle w\rangle^\perp\cap U_r$ with local coordinates
$$\varphi_R(x)=(\nn \cdot \qq,-\nn,0,-1).$$
\end{itemize} 

In the opposite direction, every proper cycle $x\neq w$ represents an oriented geometric cycle in
$\RR^n$. A proper cycle in $\langle w\rangle^\perp \subset \mathcal{U}_r$
represents a  plane, while a proper cycle in $\langle r\rangle^\perp\subset \mathcal{U}_w$
represents a point
in $\RR^n$.  The only exception is $w\in \langle w\rangle^\perp \cap \langle r\rangle^\perp$  which does not represent any geometric cycle. A change of sign
of the last homogeneous coordinate of a cycle $x\in \Omega$, produces
the {\em reoriented cycle} $x'\in \Omega$ representing the same
geometric cycle and an oriented cycle with the opposite orientation. Points have
no orientation: if $x$ is a point, then $x'=x$.

\begin{remark} Spheres and planes in $\RR^n$ correspond through the 
stereographic projection to codimension $1$ spheres on the sphere
$S^{n}$ and, in this setting, the cycle $w$ is the representation of
the pole in $S^{n}$.
\end{remark}

The Lie product computed in different charts reflects different
geometric properties of the corresponding pair of cycles. Following
are some specific cases.

\begin{enumerate}
\item Let $x_1$ and $x_2$ be proper cycles representing geometric cycles $c_1$ and $c_2$ such that $(X_1\mid X_2)=0$. If
one of the cycles, for example $c_1$, is a point, then it lies
on $c_2$. If both are non-point cycles, then $c_1$ and $c_2$ are tangent with compatible orientation. If $c_1$ and $c_2$ are both planes, then they 
are parallel with compatible orientation.  The
proof of this fact amounts to simple geometric verifications and
can be found for example in  \cite{Cecil}.

\item Let $x_1,x_2\in \mathcal{U}_r\cap \Omega$ be two non-point cycles 
representing intersecting geometric cycles $c_1$ and $c_2$. The Lie product  
\begin{equation}\label{cikli1}
(\varphi_R(x_1)\mid \varphi_R(x_2))=
-\cos\alpha-1,\end{equation}
where $\alpha$ is the angle of intersection (figure \ref{kot}, left). If both cycles are spheres
this follows from the law of cosines since  
$$(\varphi_R(x_1)\mid\varphi_R(x_2))=
\frac{-\|\mathbf{p}_1-\mathbf{p}_2\|^2+\rho_1^2+\rho_2^2-2\rho_1\rho_2}{2\rho_1\rho_2}=
-\cos\alpha-1.$$
If one cycle is a plane and one is a sphere, then  
$$(\varphi_R(x_1)\mid\varphi_R(x_2))=
\left(\varphi_R(x_1)\mid\frac{\varphi_W(x_2)}{\rho_2}\right)=
\frac{\mathbf{n}_1\cdot(\mathbf{p}_2-\mathbf{q}_1)-\rho_2}{\rho_2}=
-\cos\alpha-1,$$
where $\mathbf{q}_1$ is a point in the intersection. And finally, if both
cycles are planes, again, 
$$
(\varphi_R(x_1)\mid\varphi_R(x_2))=\mathbf{n}_1\mathbf{n}_2-1 =
-\cos\alpha-1.
$$ 
If the cycles $c_1$ and $c_2$ do not intersect, 
the Lie product is associated with the
Lorentz boost $\chi$.  
\begin{equation}\label{cikli1a}
(\phi_R(x_1)\mid\phi_R(x_2))=\pm\cosh\chi-1\quad\text{and}\quad
\cosh\chi=\left|\frac{\rho_1'^2+\rho_2'^2}{2\rho_1'\rho_2'}\right|,
\end{equation}
%
%
where $\rho_1'$ and $\rho_2'$ are the radii of the corresponding
concentric cycles (see figure \ref{kot} right) 
with the same product \eqref{cikli1a}.

\begin{figure}[h!]
\centering
\includegraphics[scale=0.85]{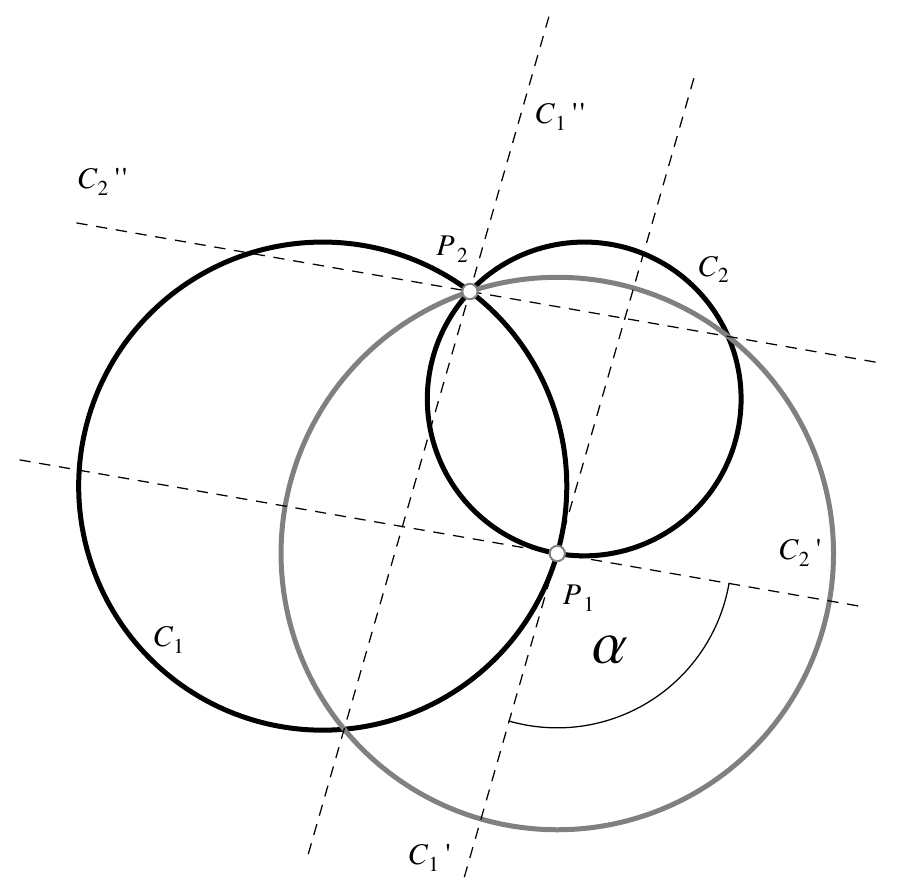}\hfill
\includegraphics[scale=0.85]{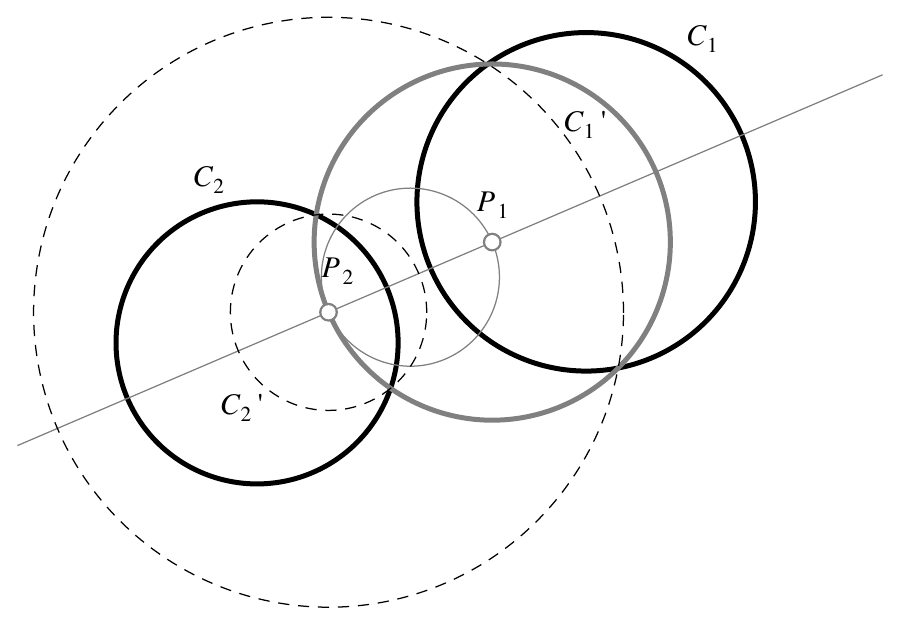}
\caption{Geometric interpretation of the Lie product in the chart $\mathcal{U}_r$ in the case of intersecting (left, equation \eqref{cikli1}), and nonintersecting (right, equation \eqref{cikli1a}) cycles.}\label{kot}
\end{figure}

\item Let $x_1,x_2\in \mathcal{U}_w\cap\Omega$ be non-point cycles representing geometric cycles $c_1$ and $c_2$ with a common tangent plane, i.e.  with 
$(\rho_1-\rho_2)^2-\|\mathbf{p}_1-\mathbf{p}_2\|^2\leq 0$ 
(where $\rho_i$ can be negative, depending on the orientation). Then 
\begin{equation}\label{cikli2}
(\varphi_W(x_1)\mid \varphi_W(x_2))=
\frac{(\rho_1-\rho_2)^2-\|\mathbf{p}_1-\mathbf{p}_2\|^2}{2}=-\frac{d^2}{2},\end{equation}
where $d$ is the tangential distance (figure \ref{potenca}, left).

\begin{figure}[h!]
\centering
\includegraphics[scale=0.85]{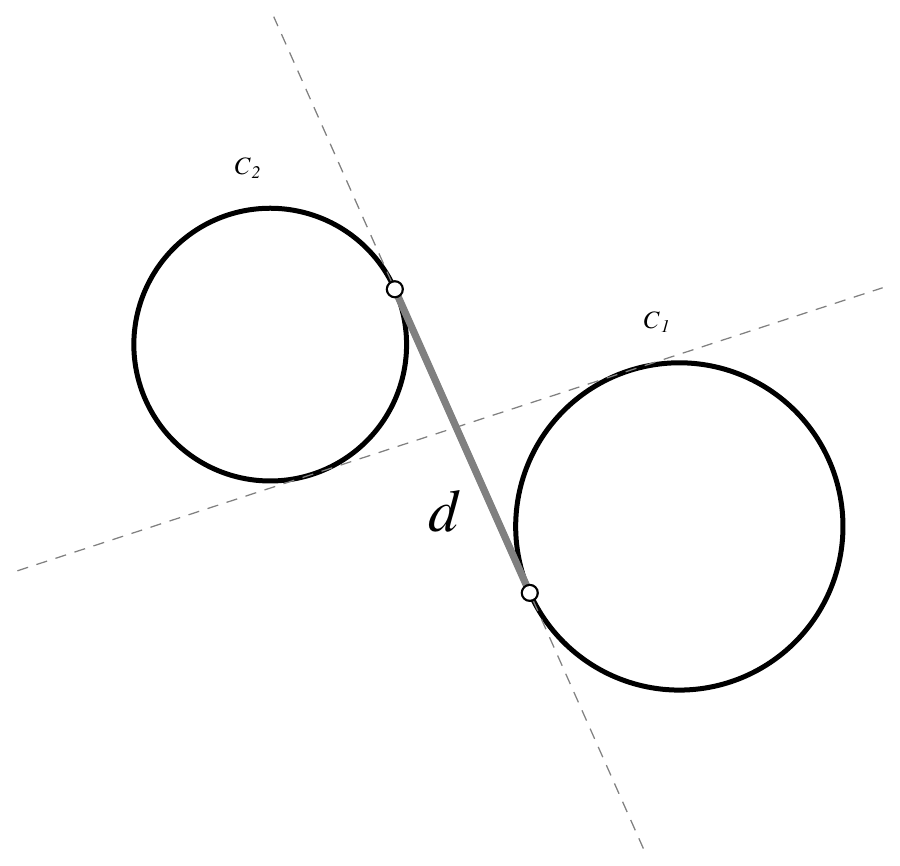}\hfill
\includegraphics[scale=0.85]{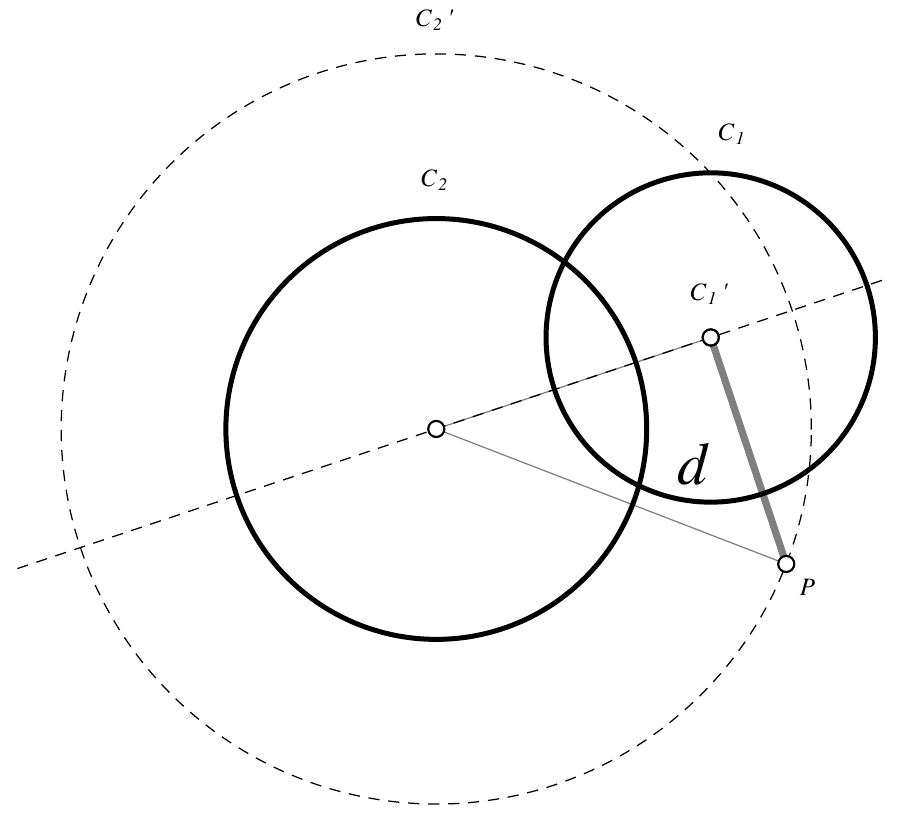}
\caption{Geometric interpretation of the Lie product in the chart $\mathcal{U}_w$. On the left, the two cycles have a common tangent plane (equation \eqref{cikli2}),  and have opposite orientations so $(\rho_1-\rho_2)$ corresponds to the sum of the two radii and $d$ is the tangential distance. On the right, the cycles do not have a common tangent plane (equation \eqref{ex}) and $d$ is the half chord.}\label{potenca}
\end{figure}
\noindent If the cycles do not have a common tangent plane, then the product
\eqref{cikli2} is positive. If $c_1$ is a point, then it lies inside the sphere $c_2$ and  equation (\ref{cikli2}) is
\begin{equation*}
(\varphi_W(x_1)\mid \varphi_W(x_2))=\frac{\rho_2^2-\|\mathbf{p}_1-\mathbf{p}_2\|^2}{2}=\frac{d^2}{2},
\end{equation*}
where $d$ is the half chord of $c_2$ through $c_1$. If both $c_1$ and $c_2$ are spheres, 
then equation (\ref{cikli2}) gives 
\begin{equation}\label{ex}(\varphi_W(x_1)\mid \varphi_W(x_2))=
\frac{(\rho_1-\rho_2)^2-\|\mathbf{p}_1-\mathbf{p}_2\|^2}{2}=\frac{d^2}{2}\end{equation} 
where $d$ is 
the half chord of the circle concentric to $c_2$ with radius $|\rho_1|+|\rho_2|$ through the center of $c_1$. 
A proof of this is a nice application of Lie reflections and is given later  in corollary \ref{poten}.
\end{enumerate}

\section{Projective subspaces and families of cycles}\label{druzine}
\subsection{Families and s-families of cycles}
In this section we will consider families of proper cycles arising from projective 
subspaces in $\PP^{n+2}$ and 
geometric objects corresponding to them.

Let $\mathbf{x}=(x_1,\ldots,x_{k+1})$, $2\le k \le n$, denote a list of cycles spanning the
subspace $\langle\mathbf{x}\rangle\subset \PP^{n+2}$. 
If the intersection $\langle\mathbf{x}\rangle\cap\Omega$ 
is nonempty it is a {\em family} of proper cycles. 
The dual algebraic object $\langle\mathbf{x}\rangle^\perp\cap\Omega$ will 
be called the \emph{cofamily}. 
On the geometric side the cofamily contains all oriented geometric cycles which 
are tangent to geometric cycles corresponding to all $x\in\mathbf{x}$. 

Typically we will consider lists of the type $(\mathbf{x},s)=(x_1,\ldots,x_k,s)$, where $x_i$ are 
proper cycles and $s$ is $r$, $w$ or possibly some other special cycle with 
with $(S\mid S)\le 0$. 
A family $\langle\mathbf{x},s\rangle$ of this type will be called an $s$-{\em family} 
and the corresponding cofamily will be an $s$-{\em cofamily}.
For example, an $r$-cofamily consists of points, and a $w$-cofamily consists of planes.  
If $k=2$, an $s$-family is usually called a \emph{pencil} of cycles. Following 
the standard terminology for pencils, we will call an $r$-family a
{\em Steiner family} and a $w$-family a {\em cone family}. 

\begin{figure}[h!]
\includegraphics[scale=0.58]{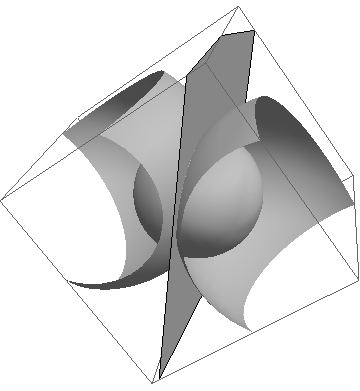}\hfill
\includegraphics[scale=0.58]{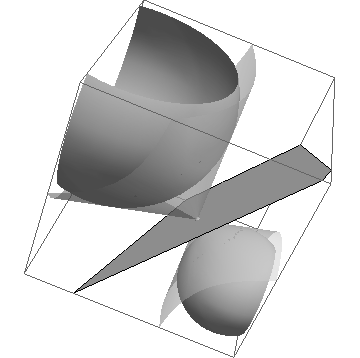}
\caption{A Steiner and a cone pencil in three dimensions}\label{steiner-cone-23}
\end{figure}

Let us take a closer look at geometric pencils in $\RR^3$ (see also \cite{Y} and \cite{BJZNMK1}) . 
A Steiner pencil is determined by two intersecting cycles  
and consists of all cycles containing the intersection circle or line
of these two cycles, so the geometric object it represents is the intersecting circle. 
The cycles of the copencil are the points of this circle. 
A cone pencil is determined by two oriented spheres with 
a common tangent plane.
The cycles of the copencil are the planes tangent (with compatible orientations) to both spheres.  
The envelope of these planes is a double cone tangent to all the
spheres in the pencil with vertex in the common intersection point of the tangent planes, so the geometric
object represented by a cone family is an oriented double cone (which collapses to a line if both spanning cycles are points, or
extends to a cylinder if the two spanning cycles are
spheres with equal radii). 
Another interesting geometric pencil is obtained if the special cycle $s$ is given by $S=\rho W+R$. Then an $s$-pencil spanned by two oriented spheres consists of all spheres with a common
tangent sphere of radius $\rho$. The cycles of the copencil are all spheres 
of radius $\rho$, tangent to both given spheres. The envelope of these spheres forms a torus, so the new geometric object represented by such a pencil is an oriented torus or, if the two spanning cycles are planes, a cylinder. 

\subsection{The determinant of a family}
Consider the Gram matrix with elements the Lie products of vectors from a list of linearly independent vectors $\mathbf{X}=(X_1,\ldots,X_k)$,
\begin{equation}\label{restricted}
\mathbf{A}_{\mathbf{X}}=[\mathbf{X}]^T \mathbf{A} [\mathbf{X}]=\left[\begin{array}{ccc} (X_1\mid X_1) & \cdots & (X_1\mid X_k)\\
\vdots & \ddots & \vdots\\
(X_k\mid X_1) & \cdots & (X_k\mid X_k)\end{array}\right].
\end{equation}
Its determinant 
$\Delta(\mathbf{X})=\det \mathbf{A}_\mathbf{X}$ 
can be positive, negative or even zero, due to the indefiniteness of the Lie product. The sign of the determinant $\Delta(\mathbf{X})$  
is an invariant of the projective subspace 
$\langle\mathbf{x}\rangle$ and consequently an invariant of 
the underlying family $\langle\mathbf{x}\rangle\cap\Omega$.  
It depends on the index 
of the Lie form restricted to the subspace $\langle\mathbf{X}\rangle$. 
The following proposition is a standard result in linear algebra, for example, it is an immediate consequence of 
\cite[Theorem 1.2]{Cecil}.

\begin{proposition}\label{signature-1}
\begin{enumerate}
\item If the Lie form on  
the subspace $\langle\mathbf{X}\rangle$ is 
nondegenerate then it is nondegenerate also on the Lie
orthogonal complement  $\langle\mathbf{X}\rangle^\perp$, and 
$
	\RR^{n+3}=\langle\mathbf{X}\rangle\oplus\langle\mathbf{X}\rangle^\perp.
$ 
\item If $\Delta(\mathbf{X})<0$, then the index of the Lie form on the
subspace $\langle \mathbf{X}\rangle$ is $1$.
\end{enumerate}
\end{proposition}

The sign of $\Delta(\mathbf{X})$ determines the position of the 
projective subspace $\langle \mathbf{x}\rangle$ 
and its Lie orthogonal complement with respect to the Lie quadric $\Omega$, as the following theorem shows.

\begin{theorem}\label{proj-family}\ \\
	Let $\mathbf{X}=(X_1,\dots X_k)$, where $2\le k \le n+1$ be  
	linearly independent vectors from $\mathbb{R}^{n+3}$.
	\begin{enumerate}
	\item If $\Delta(\mathbf{X})<0$, then both $\langle\mathbf{x}\rangle$ and its 
	Lie orthogonal complement $\langle\mathbf{x}\rangle^\perp$ intersect $\Omega$.
	\item If $\Delta(\mathbf{X})=0$, then 
	$\langle\mathbf{x}\rangle\cap\langle\mathbf{x}\rangle^\perp\ne\emptyset\subset\Omega$. 
	\item If $\Delta(\mathbf{X})>0$, then exactly one of $\langle\mathbf{x}\rangle$ and 
	$\langle\mathbf{x}\rangle^\perp$ does not intersect $\Omega$. 
	\end{enumerate}
\end{theorem}
\begin{proof}
	\begin{enumerate}
		\item If $\Delta(\mathbf{X})<0$
		the restriction of the Lie form to $\langle\mathbf{X}\rangle$ is
		nondegenerate and has index $1$, so there exists a Lie orthogonal basis of 
		$\langle\mathbf{X}\rangle$
		formed by vectors $Y_i$ such that  $(Y_i\mid Y_i)>0$, $i=1,\dots,k-1$ and 
		$(Y_k\mid Y_k)<0$, (compare \cite[Theorem 1.2]{Cecil}). 
		So the projective line $\langle y_i,y_k\rangle$, $1\le i<k$, intersects the 
		quadric $\Omega$. The same is true for Lie orthogonal complement 
		$\langle\mathbf{X}\rangle^\perp$. 
		\item If  $\Delta(\mathbf{X})=0$, the Lie form is degenerate 
		on the subspace  
		$\langle \mathbf{X} \rangle$ and there exists a vector $X$
		Lie orthogonal to all vectors of the subspace including
		itself, so 
		$x\in\langle\mathbf{x}\rangle^\perp\cap\langle\mathbf{x}\rangle\subset\Omega$.   
		\item Finally, in the case $\Delta(\mathbf{X})>0$ we have two possibilities. 
		Either the index of the Lie form is $2$  on 
		$\langle\mathbf{X}\rangle$ and $0$ on $\langle\mathbf{X}\rangle^\perp$ and 
		$\langle\mathbf{X}\rangle$ intersects $\Omega$ 
		while $\langle\mathbf{X}\rangle^\perp$ 
		does not, or the other way around. 
	\end{enumerate} 
\end{proof}

An immediate application of theorem \ref{proj-family} is a necessary and 
sufficient condition for the existence of solutions of the oriented Apollonius problem in $\RR^n$ which asks for 
an oriented geometric cycle tangent (with compatible orientations) to $n+1$ given oriented cycles $\mathbf{c}=(c_1,\ldots,c_{n+1})$. 

\begin{coro} The oriented Apollonius problem on a configuration of $n+1$ 
oriented cycles $\mathbf{c}$ with corresponding algebraic cycles $\mathbf{x}$ which
are represented by the linearly independent set of homogeneous coordinate vectors $\mathbf{X}$, has 
exactly two solutions if $\Delta(\mathbf{X})<0$, no solution if $\Delta(\mathbf{X})>0$, 
and one solution if $\Delta(\mathbf{X})=0$.
\end{coro}
\begin{proof}
A solution of the Apollonius problem is given by  
an intersection of $\langle\mathbf{x}\rangle^\perp$ with 
the Lie quadric. 
If  $\Delta(\mathbf{X})<0$ then 
$\langle\mathbf{x}\rangle^\perp$ is a projective line which intersects the
Lie quadric in two cycles corresponding to the two solutions. 
If $\Delta(\mathbf{X})>0$ then, since $\langle\mathbf{x}\rangle$ intersects the quadric its 
dual $\langle\mathbf{x}\rangle^\perp$ does not, and the problem has no solutions. 
Finally, if $\Delta(\mathbf{X})=0$ then the
projective line $\langle\mathbf{x}\rangle^\perp$ is tangent 
to  $\Omega$, so it contains 
exactly one point from $\Omega$, so the problem has exactly one solution.
\end{proof}

\subsection{Lie projections and Lie reflections}
Let $\mathbf{X}=(X_1,\ldots,X_k)\in \RR^{n+3}$ be a $k$-tuple of linearly
independent vectors such that $\Delta(\mathbf{X})\neq 0$ (that is, with $\mathbf{A_X}$ nonsingular). 
The {\em Lie orthogonal projection} 
onto the subspace $\langle\mathbf{X}\rangle$
is given by
\begin{equation}\label{projektor}
P_{\langle\mathbf{X}\rangle}Y =
[\mathbf{X}]\mathbf{A}_\mathbf{X}^{-1}[\mathbf{X}]^T \mathbf{A} Y.\end{equation}

The Lie orthogonal projection onto $\langle\mathbf{X}\rangle$ determines a projective map 
$$P_{\langle\mathbf{x}\rangle}\:
\PP^{n+2}\setminus\left<\mathbf{x}\right>^\perp\to\mathbf{x}\subset \PP^{n+2}.$$

\begin{proposition}\label{projektorji}
The Lie orthogonal projection has the following properties:
\begin{enumerate}
\item $P_{\langle\mathbf{X}\rangle^\perp}= \mbox{\rm Id} - P_{\langle\mathbf{X}\rangle}$.
\item  For any $Y_1,Y_2,Y\in \RR^{n+3}$, 
  $$(Y_1\mid P_{\langle\mathbf{X}\rangle}Y_2)=(P_{\langle\mathbf{X}\rangle}Y_1\mid
  Y_2)\quad\mbox{and}\quad(Y\mid P_{\langle\mathbf{X}\rangle}Y) =
  (P_{\langle\mathbf{X}\rangle}Y\mid P_{\langle\mathbf{X}\rangle}Y).$$
  \end{enumerate}
\end{proposition}

\begin{proof}
\begin{enumerate}
\item Since $\Delta(\mathbf{X})\neq 0$, the Lie form is nondegenerate on $\mathbf{X}$ and, by lemma \ref{signature-1}, each vector $Y\in \RR^{n+3}$ has a unique decomposition 
into $Y=Y_1+Y_2$, with $Y_1\in \langle\mathbf{X}\rangle$ and $Y_2\in \langle\mathbf{X}\rangle^\perp$. Then 
$P_{\langle\mathbf{X}\rangle}(Y)=Y_1$ and $P_{\langle\mathbf{X}\rangle^\perp}(Y)=Y_2$, and so
$P_{\langle\mathbf{X}\rangle^\perp}(Y)=Y-P_{\langle\mathbf{X}\rangle}(Y)$.
\item Since $\mathbf{A}$, as well as  $\mathbf{A}_{\mathbf{X}}=([\mathbf{X}]^T \mathbf{A} [\mathbf{X}])$, is symmetric 
\begin{eqnarray*}
(Y_1\mid P_{\langle\mathbf{X}\rangle}Y_2)
&=& ([\mathbf{X}]\mathbf{A}_{\mathbf{X}}^{-1}[\mathbf{X}]^T\mathbf{A}Y_2)^T 
\mathbf{A}Y_1\\
&=&Y_2^T\mathbf{A}([\mathbf{X}]\mathbf{A}_{\mathbf{X}}^{-1}[\mathbf{X}]^T)%
\mathbf{A}Y_1\\
&=&Y_2^T\mathbf{A}P_{\langle\mathbf{X}\rangle}Y_1\\
&=&\left(P_{\langle\mathbf{X}\rangle}Y_1\mid Y_2\right)
\end{eqnarray*}
The second statement follows from the first, since 
$P_{\langle\mathbf{X}\rangle}^2=P_{\langle\mathbf{X}\rangle}$. 
\end{enumerate}
\end{proof}

\begin{proposition}\label{sprbaze}
Let $\mathbf{X}=(X_1,\ldots,X_k)$, $\mathbf{Y}=(Y_1,\ldots,Y_m)$, and 
$\mathbf{Z}=(X_1,\ldots,X_k,Y_1,\ldots,Y_m)$ be linearly independent vectors and $\Delta(\mathbf{X})\ne 0$. Then  
\begin{equation}\label{sprbazef}
\Delta(\mathbf{Z})
=\Delta(\mathbf{X})\Delta(P_{\langle\mathbf{X}\rangle^\perp}\mathbf{Y}).
\end{equation}\label{razbiti-na-produkt}
\end{proposition}

\begin{proof}
The determinant of a block matrix is 
$$\det \left[\begin{array}{cc}\mathbf{A} & \mathbf{C}\\ \mathbf{D} & \mathbf{B}\end{array}\right]=\det \mathbf{A} \det\left(\mathbf{B}-\mathbf{D}\mathbf{A}^{-1}\mathbf{C}\right),$$
and so 
\begin{eqnarray*}\Delta(\mathbf{Z}) &=& \det \left[\begin{array}{cc}
[\mathbf{X}]^T\mathbf{A}[\mathbf{X}] & [\mathbf{X}]^T\mathbf{A}[\mathbf{Y}]\cr 
[\mathbf{Y}]^T
\mathbf{A}[\mathbf{X}] & [\mathbf{Y}]^T
\mathbf{A}[\mathbf{Y}]\end{array}\right]\\
&=& \mbox{det}\,([\mathbf{X}]^T\mathbf{A}[\mathbf{X}]))
\cdot\mbox{det}\,([\mathbf{Y}]^T\mathbf{A}[\mathbf{Y}]- 
[\mathbf{Y}]^T\mathbf{A}[\mathbf{X}] ([\mathbf{X}]^T\mathbf{A}[\mathbf{X}])^{-1} [\mathbf{X}]^T\mathbf{A}
[\mathbf{Y}])\\
&=& \Delta(\mathbf{X})\cdot\mbox{det}\,([\mathbf{Y}]^T\mathbf{A}([\mathbf{Y}]-P_{\langle\mathbf{X}\rangle}[\mathbf{Y}]))\\
&=& \Delta(\mathbf{X})\cdot\mbox{det}\,([\mathbf{Y}]^T\mathbf{A}P_{\langle\mathbf{X}\rangle^\perp}[\mathbf{Y}])\\
&=& \Delta(\mathbf{X})\cdot\mbox{det}\,([P_{\mathbf{X}^\perp}(\mathbf{Y})]^T\mathbf{A}[P_{\langle\mathbf{X}\rangle^\perp}
(\mathbf{Y})])\\
&=& \Delta(\mathbf{X})\Delta(P_{\langle\mathbf{X}\rangle^\perp}(\mathbf{Y}))
\end{eqnarray*}
\end{proof}

The {\em Lie reflection} with respect to the subspace $\left<\mathbf{X}\right>$
is given by 
$$
L_{\langle\mathbf{X}\rangle}Y = Y - 2P_{\langle\mathbf{X}\rangle^\perp}Y=
P_{\langle\mathbf{X}\rangle}Y-P_{\langle\mathbf{X}\rangle^\perp}Y.
$$
If $\Delta(\mathbf{X})\neq 0$ the reflection $L_{\langle\mathbf{X}\rangle}$ is an isomorphism and determines 
the projective map 
$$L_{\langle\mathbf{x}\rangle}\: \PP^{n+2}\to \PP^{n+2}$$ 
which is the identity on 
$\left<\mathbf{x}\right>$ and on $\langle\mathbf{x}\rangle^\perp$.\\
It has the following obvious properties:
\begin{enumerate}
\item
$L_{\langle\mathbf{X}\rangle}=-L_{\langle\mathbf{X}\rangle^\perp}$,
\item if $Y\in \left<\mathbf{X}\right>$ then $L_{\langle\mathbf{X}\rangle}Y=Y$ and if $Y\in \left<\mathbf{X}\right>^\perp$ then $L_{\langle\mathbf{X}\rangle}Y=-Y$,
\item $(L_{\langle\mathbf{X}\rangle}Y\mid
L_{\langle\mathbf{X}\rangle}Z)=(Y\mid Z)$,
\item if $S\in \left<\mathbf{X}\right>$ 
then 
$$(S\mid Y)=-(L_{\langle\mathbf{X}\rangle}S\mid Y) = -(S\mid
L_{\langle\mathbf{X}\rangle}Y),$$
\item if
$S\in\left<\mathbf{X}\right>$ then
$$(S\mid Y)=(L_{\langle\mathbf{X}\rangle}S\mid Y)=(S\mid
L_{\langle\mathbf{X}\rangle}Y)$$
and the reflection
$L_{\langle\mathbf{X}\rangle}$ preserves the chart
$\mathcal{U}_s$
and local coordinates in it,
since
$$
L_{\langle\mathbf{X}\rangle}\varphi_S(y)=
L_{\langle\mathbf{X}\rangle}\left(\frac{1}{(S\mid Y)}Y\right)= 
\frac{1}{(Y\mid S)}L_{\langle\mathbf{X}\rangle} \, Y=
\frac{1}{(
L_{\langle\mathbf{X}\rangle} \,Y
\mid S)}L_{\langle\mathbf{X}\rangle} \, Y=
\varphi_S(
L_{\langle\mathbf{x}\rangle}y).
$$
\end{enumerate}
Here are some interesting particular cases:
\begin{enumerate}
	\item If either $r\in \langle\mathbf{x}\rangle$ or $r\in \langle\mathbf{x}\rangle^\perp$, the  reflection $L_{\langle\mathbf{x}\rangle}$ 
	preserves angles. 
	\item	If either $w\in \langle\mathbf{x}\rangle$ or $w\in \langle\mathbf{x}\rangle^\perp$, the reflection $L_{\langle\mathbf{x}\rangle}$ preserves tangential distance.
	\item If $x\in \langle{r}\rangle^\perp\setminus\Omega$ is a nonproper cycle,  
	then $L_x$ represents the geometric inversion across the proper cycle 
	$\langle{x,r}\rangle\cap\Omega$ (if it exists). 
	\item The reflection $L_{\langle r\rangle}$ reverses orientation of cycles.
	\item If the special cycle 
$s$ is of the form $S=\rho W+R$, $\rho\in \RR$, then $L_{s}$ reverses orientation of spheres and changes the radius 
by $\frac{1}{2\rho}$. Since $s\in \langle{w}\rangle^\perp$ it also preserves 
	tangential distances. Points are thus mapped to spheres of radius $\frac{1}{2\rho}$, while spheres of 
	radius $\frac{1}{2\rho}$ are mapped to points. 
	On planes $L_{s}$ represents a parallel shift and 
	reverses orientation. 
	\end{enumerate}  
The reflection $L_{s}$ from the last example above can be used to give a nice proof of what is the Lie product of two spheres 
in the case $(\varphi_W(x)\mid \varphi_W(y))>0$ (figure \ref{potenca}): 

\begin{coro}\label{poten}
In the case when $x$ and $y$ represent spheres $c_x$ and $c_y$ with radii $\rho_x$ and $\rho_y$ with no common tangent plane, the product $(\varphi_W(x)\mid \varphi_W(y))$ 
is $d^2/2$, where $d$ is the half chord of a sphere concentric to $c_y$ 
with radius $\rho_x+\rho_y$ through the center of $c_y$. 
\end{coro}

\begin{proof} The reflection $L_{s}$, where $S=\frac{1}{\rho_x}W+R$ preserves the chart
$\mathcal{U}_w$ and local coordinates in it, so 
$$(\varphi_W(x)\mid \varphi_W(y))=(\varphi_W(L_{s}x)\mid \varphi_W(L_{s}y))=\frac{d^2}{2},$$
where $L_{s}(x)$ is a point cycle representing the center of $c_x$, and $z=L_{s}(y)$ 
represents the concentric sphere $c_{z}$ to $c_y$ with radius $\rho_x+\rho_y$, so $d$ 
is the half chord of $c_z$ through the center of $c_x$. 
\end{proof}

\subsection{Hyperbolic $s$-families}
Let $\langle\mathbf{x},s\rangle\cap\Omega$ be an $s$-family given by linearly independent vectors $\mathbf{X}=(X_1,\dots,X_k)$, $2\le k\le n$ and $S$.    
Depending on the sign of determinant $\Delta(\mathbf{X},S)$ 
we distinguish three types of $s$-families
$\langle\mathbf{x},s\rangle\cap\Omega$.
\begin{enumerate}
	\item If $\Delta(\mathbf{X},S)>0$ the family is 
	\emph{elliptic}, 
	\item if $\Delta(\mathbf{X},S)<0$ the family is \emph{hyperbolic} and
	\item if $\Delta(\mathbf{X},S)=0$ the family is \emph{parabolic}. 
\end{enumerate} 
We will be interested only in hyperbolic families, 
since, according to theorem \ref{proj-family} a hyperbolic family as well as its cofamily are both nonempty 
and thus determine geometric objects. 

\begin{definition} Let $x$ be a nonproper cycle. 
An intersection of the projective line $\langle x,s\rangle$ with $\Omega$ 
will be called a {\em projection of $x$ onto $\Omega$ along $s$}. \end{definition}

\begin{lema}If $(S\mid S)\neq 0$ then a nonproper cycle $x$ has two different projections onto $\Omega$ along $s$ if $\Delta(X,S)<0$, one projection if $\Delta(X,S)=0$ and no projections if $\Delta(X,S)>0$. If $(S\mid S)=0$ then one of the projections is equal to $s$, and a second one exists if and only if $\Delta(X,S)\neq 0$.
\end{lema}

\begin{proof} If $(S\mid S)\neq 0$ then the equation 
\begin{equation}\label{nn}(X+\lambda S\mid X+\lambda S)=(X\mid X)+2\lambda(X\mid S)+\lambda^2(S\mid S)=0\end{equation}
is quadratic with discriminant $-\Delta(X,S)$, so the claim follows.
If $(S\mid S)=0$ then $s$ is automatically the projection of an arbitrary $x$ onto $\Omega$ along $s$. A second projection exists if $(X\mid S)\neq 0$, it is the solution of equation (\ref{nn}) which is linear in this case. If $(X\mid S)=0$, equation (\ref{nn}) has no solutions.
\end{proof}
 
 \begin{theorem}\label{discr-family}
	Let $\langle\mathbf{x},s\rangle\cap \Omega$ be a hyperbolic $s$-family.  	
	\begin{enumerate} 
		\item If  $(S\mid S)<0$, 
		then each nonproper cycle $x\in\langle\mathbf{x},s\rangle$ different from $s$  
		has two projections onto $\Omega$ along $s$. Each cycle $x\in\langle\mathbf{x},s\rangle\cap\langle s \rangle^\perp$ is nonproper with 
		$(X\mid X)>0$. The  subspace $\langle\mathbf{X},S\rangle$ is a direct sum 
		\begin{equation}\label{direct}
		\langle\mathbf{X},S\rangle=\langle\mathbf{X},S\rangle\cap\langle
		S \rangle^\perp\oplus\langle S\rangle.
		\end{equation} 		
		\item If $(S\mid S)=0$ then a cycle 
		$x\in \langle\mathbf{x},s\rangle\cap\langle s \rangle^\perp$ 
		different from $s$ has no projection onto $\Omega$ along $s$ different from $s$, and is  nonproper with $(X\mid X)>0$. The subspace 
		$\langle\mathbf{X},S\rangle$ is a direct sum 
		\begin{equation}\label{direct1}\langle\mathbf{X},S\rangle=\langle\mathbf{X},S\rangle\cap\langle S\rangle^\perp
                \oplus\langle Y\rangle\end{equation}
                where $y$ is any proper cycle in $\langle\mathbf{x},s\rangle$ not in 
		$\langle s\rangle^\perp$.  
	\end{enumerate}
\end{theorem}
\begin{proof} 
	\begin{enumerate}
       		\item Let $(S\mid S)<0$. The restriction of the Lie form to $\langle \mathbf{x},s\rangle$ has index $1$. For any $x\in \langle\mathbf{x},s\rangle$ where $x$ is nonproper and different from $s$ the same is true also for the subspace $\langle x,s\rangle$, so $\Delta(X,S)<0$, and by the lemma $x$ has two projection onto $\Omega$ along $s$.  If $x\in \langle\mathbf{x},s\rangle \cap \langle s\rangle^\perp$, 
		the projection of $x$ onto $\Omega$ along $s$ has homogeneous coordinates $X+\lambda S$ satisfying the equation 
		$$(X+\lambda S\mid X+\lambda S)=(X\mid X)+\lambda^2(S\mid S),$$
		so $(X\mid X)=-\lambda^2(S\mid S)>0$ and $x$ is non proper. Clearly, equation (\ref{direct}) holds in this case.
		
		\item Let  $(S\mid S)=0$, in this case 
		$s\in \langle\mathbf{x},s\rangle\cap\langle s \rangle^\perp$.
		For each cycle $x\in
		\langle\mathbf{x},s\rangle\cap\langle s\rangle^\perp$
		$$\Delta(X,S)=(X\mid X)(S\mid S)-(X\mid S)^2=0$$
		so by the lemma there is no projection of $x$ onto $\Omega$ along $s$ different from $s$. Let $y\in\langle \mathbf{x},s\rangle$ be any fixed proper cycles different from $s$. For any $y'\neq y,s$, the difference $X=\varphi_S(y)-\varphi_S(y')$ is in $\langle S\rangle^\perp$, so $Y'=X+Y$ and equation (\ref{direct1}) is valid. Finally, for any $x\in\langle\mathbf{x},s\rangle\cap\langle s\rangle^\perp$ different from $s$ let $y'=L_xy$ and let $X=\varphi_S(y)-\varphi_S(y')$. Then 
		$$(X\mid X)=(\varphi_S(y)-\varphi_S(y')\mid \varphi_S(y)-\varphi_S(y'))=-2(\varphi_S(y)\mid\varphi_S(y'))>0.$$
			
	\end{enumerate}
\end{proof}

\subsection{Determinants and geometry}\label{simplex}

Recall the following well known determinant from geometry in $\RR^n$. 
Let $a_0,a_1,\dots,a_k$ be the vertices of a  
$k$-dimensional simplex in $\RR^n$. Then the volume of the simplex is given by the {\em Cayley-Menger determinant} 
\begin{equation} \label{Cayley-Menger}
	\mbox{\rm{vo}l}(a_0,\ldots,a_k)^2=\frac{(-1)^{k-1}}{2^k\,k!^2}\,
        \left| \begin{array}{cccccc}
        0 & d^2_{01}& d^2_{02} & \cdots & d^2_{0k} & 1\\
        d^2_{10} & 0 & d^2_{12}&\cdots & d^2_{1k} & 1\\
        \vdots & \vdots & \vdots & \ddots & \vdots & \vdots\\
        d^2_{k0} & d^2_{k1} & d^2_{k2}&\cdots & 0 & 1 \\
         1 & 1 & 1 & 1 & 1 & 0 \end{array} \right|
\end{equation}
where $d_{ij}=|a_i-a_j|$. 
The \emph{polar sine} of the angle at the vertex $a_0$ of 
the simplex $(a_0,a_1,\dots,a_k)$ is defined as 
\begin{equation}\label{Polar-sine}
\psin_{a_0}(a_0,a_1,\dots,a_k)=k!\frac{\vol(a_0,a_1,\dots,a_k)}{|a_1-a_0|\dots|a_k-a_0|}.
\end{equation}
It is the ratio of the volume of the given simplex
by the volume of the cube
with edges from the vertex $a_0$ of the same lengths. 

Determinants of $r$-families and $w$-families computed in different charts have several interesting geometric interpretations. Following are some examples.
\begin{enumerate}
\item Let $\mathbf{x}=(x_1,\dots,x_k)$ be independent proper cycles, 
where $2\le k \le n-1$ and 
$X_i=\varphi_W(x_i)$ the homogeneous coordinates in the chart $\mathcal{U}_w$. 
\begin{itemize}

\item The determinant  $\Delta(\mathbf{X},W)$ equals
\begin{equation}\label{Vol}
   \Delta(\mathbf{X},W)=
        \left| \begin{array}{cccccc}
        0 & -\frac{d^2_{12}}{2} & -\frac{d^2_{13}}{2} & \cdots & -\frac{d^2_{1k}}{2} & 1\\
        -\frac{d^2_{21}}{2} & 0 &-\frac{d^2_{23}}{2} & \cdots & -\frac{d^2_{2k}}{2} & 1\\
        \vdots & \vdots &\vdots &\ddots & \vdots & \vdots\\
        -\frac{d^2_{k1}}{2} & -\frac{d^2_{k2}}{2} & -\frac{d^2_{k3}}{2} & \cdots & 0 & 1 \\
         1 & 1 & 1 & 1 & 1 & 0 \end{array} \right|
\end{equation}
where $d_{ij}$ is the tangential distance between the cycle $x_i$ and $x_j$.  
If $\Delta(\mathbf{X},W)<0$ then by (\ref{Cayley-Menger})
$$
	\Delta(\mathbf{X},W) =-(k-1)!^2 \vol(q_1,\ldots, q_k)^2
$$
where $q_1,\ldots,q_k$ are the vertices of the {\em contact simplex}, that is, the points of tangency of the cycles of $\mathbf{x}$ to their common tangent  plane. 
The sides of this simplex have lengths equal to the tangential distances 
between the cycles of $\mathbf{x}$ (Figure \ref{simpleksi}, left). 

\begin{center}
\begin{figure}
\includegraphics[width=7cm]{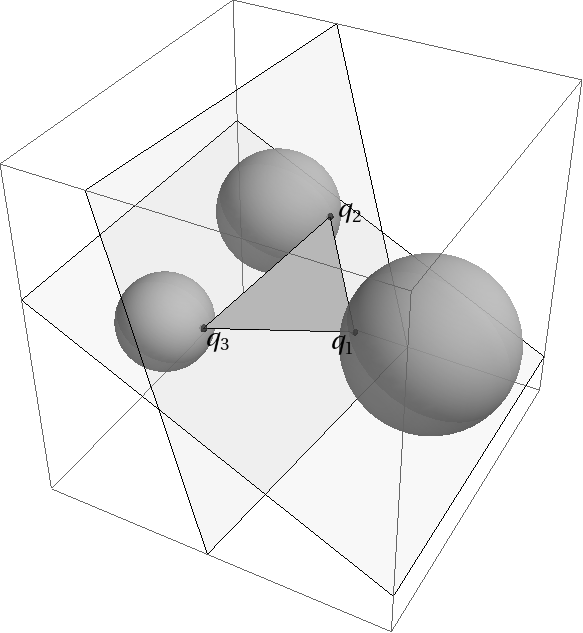}\qquad
\includegraphics[width=7cm]{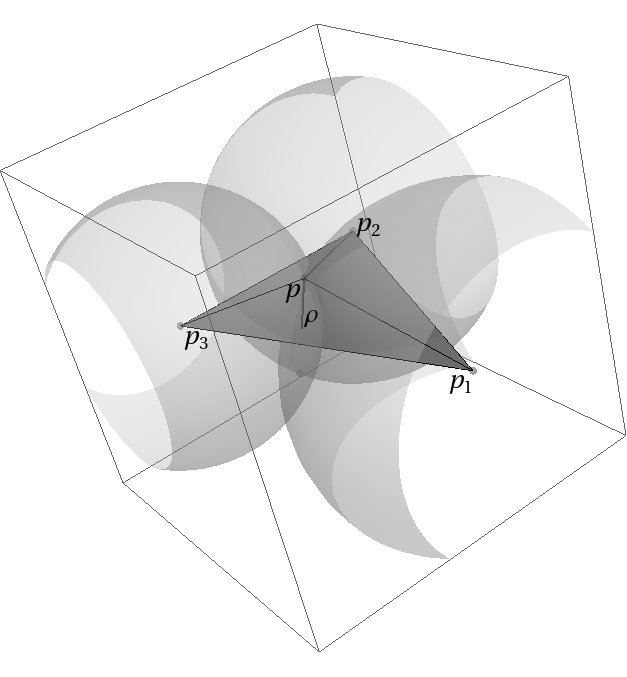}
\caption{The contact simplex (left) and the simplex with vertices in the centers of the cycles (right)\label{simpleksi}}
\end{figure}
\end{center}

\item  The determinant 
$\Delta(\mathbf{X},R,W)$ equals  by proposition \ref{sprbaze} to
\begin{equation}\label{Vol0}
\Delta(\mathbf{X},R,W)=(R\mid R)\Delta(P_{\langle R\rangle^\perp}\mathbf{X},W)=-\Delta(P_{\langle R\rangle^\perp}\mathbf{X},W)
\end{equation}
where  $P_{\langle R\rangle^\perp}\mathbf{X}$ is formed by the coordinate vectors in the $W$-chart of the point cycles $p_i$ corresponding to
the centers of $x_i$, so  
\begin{equation}\label{kren}\Delta(\mathbf{X},R,W) = (k-1)!^2\vol(p_1,\ldots,p_k)^2,\end{equation}
gives the squared volume of the simplex with vertices in the centers $p_i$ of the cycles $x_i$ (Figure \ref{simpleksi}, right).

\item Finally, let $(x_1,\ldots, x_k)$ be a $k$-tuple of intersecting proper cycles, $p$ a point in the common intersection, and $P$ its coordinate vector in the 
$W$-chart. Then 
\begin{equation*}
	\Delta(\mathbf{X},P,R,W)=
\left|\begin{array}{cccccc}
        0 & (X_1\mid X_2) & \cdots & 0 & (X_1\mid R) & 1\\
        (X_2\mid X_1) & 0 &\cdots & 0 & (X_2\mid R) & 1\\
        \vdots & \vdots &\ddots &\vdots & \vdots & \vdots\\
        0 & 0 & 0  & \cdots & 0 & 1 \\
	(R\mid X_1) & (R\mid X_2) &\cdots & 0 & -1 & 0 \\
	1 & 1 &\cdots & 1 & 0 & 0
	\end{array} \right| = \Delta(\mathbf{X},R).
\end{equation*}
and 
\begin{equation}\label{Vol1}
\Delta(\mathbf{X},P,R,W)=(R\mid R)\Delta(P_{\langle
r\rangle^\perp}(\mathbf{X},P)),W)
=-\Delta(P_{\langle
r\rangle^\perp}\mathbf{X},P,W)
\end{equation}
where $P_{\langle
r\rangle^\perp}\mathbf{X}$ represents the point cycles $(p_1,\ldots,p_k)$ of the centers of $x_i$ expressed in the $W$-chart.
By (\ref{kren}),  
\begin{equation}\label{prejsnji}\Delta(\mathbf{X},P,R,W)=\Delta(\mathbf{X},R)=-k!^2\vol(p_1,\ldots,p_k,p)^2\end{equation}
where $\vol(p_1,\ldots,p_k,p)$ is the volume of the simplex with vertices in the centers $p_i$ of the cycles $x_i$ and a point $p$ in the intersection. 

\end{itemize}
\item Let $\mathbf{x}=(x_1,\dots,x_k)$, $2\le k \le n-1$, be independent proper cycles representing spheres and 
let $X_i=\varphi_R(x_i)$ be homogeneous coordinates in the chart $\mathcal{U}_r$. 
Then  $X_i=\varphi_R(x_i)=\varphi_W(x_i)/\rho_i$ where $\rho_i$ is the radius of the sphere $x_i$ and by (\ref{Polar-sine}) and (\ref{prejsnji}),
\begin{equation}\label{StY}
\Delta(\mathbf{X},R)=-\frac{\Delta(\varphi_W(\mathbf{x}),R)}{\rho_1^2 \dots \rho_k^2}=
-\frac{k!^2\vol(p,p_1,\ldots,p_k)^2}{\rho_1^2\dots \rho_k^2}=-\psin_p(p,p_1,\ldots,p_k)^2,
\end{equation}
On the other hand 
$$(P_{\langle R\rangle^\perp}X_i\mid P_{\langle R\rangle^\perp}X_j)=(X_i\mid X_j)+1 =- \cos \psi_{ij}$$ 
$$ \Delta(\mathbf{X},R)=(R\mid R)\Delta(P_{\langle R \rangle^\perp}\mathbf{X})=-\Delta(P_{\langle R \rangle^\perp}\mathbf{X})$$
it follows that 
\begin{equation}
	\Delta(\mathbf{X},R)=
	-\psin(p,p_1,\ldots,p_k)^2=
	-\left| \begin{array}{cccc}
        1 & -\cos\psi_{12} &  \cdots & -\cos\psi_{1k}\\
        -\cos\psi_{21} & 1 &  \cdots & -\cos\psi_{2,k}\\
        \vdots & \vdots &  \ddots & \vdots\\
        -\cos\psi_{k1} & -\cos\psi_{k2} &  \cdots & 1\\
	 \end{array} \right|
\end{equation}
where $\psi_{ij}$ is the intersection angle of the spheres 
$x_i$ and $x_j$.
\end{enumerate}

\subsection{Discriminants}
Since the value
of the determinant of a family depends on the choice of coordinate vectors, 
it can not, by itself, reflect geometric properties of the 
underlying geometric objects like the radius, 
center and orientation of a circle given by a Steiner 
pencil or the  vertex, axis and angle at the
vertex of a cone given by a cone pencil. These are, however, reflected by   
quotients of determinants. For example, the square 
of the radius of a sphere given by $x$ is equal to 
$$\rho^2=-\frac{\Delta(X,R)}{\Delta(X,R,W)}.$$

\begin{definition} Let $S'$ be such that $(\mathbf{X},S,S')$ are linearly independent, $\Delta(\mathbf{X},S)\neq 0$, and $(S'\mid S)=0$. 
The {\em $S'$-discriminant} of $\langle\mathbf{x},s\rangle$ is given by 
\begin{equation}\label{definicija-delta}
        \delta_{S'}(\mathbf{x},s)=\frac{\Delta(\mathbf{X},S,S')}{\Delta(\mathbf{X},S)}
\end{equation}
\end{definition}
Note that the value of $\delta_{S'}(\mathbf{x},s)$  depends not only on the cycle $s'$ but also on the choice of the vector $S'$.  

For example, if $\mathbf{x}=x$ is a single cycle, the $S'$-discriminant of $(x,s)$ is 
\begin{eqnarray*}\delta_{S'}(x,s)&=&\frac{\Delta(X,S,S')}{\Delta(X,S)}=-\frac{1}{(X\mid X)(S\mid S)-(X\mid S)^2} \, \left|\begin{array}{ccc} 
(X\mid X) & (X\mid S) & (X\mid S')\\ (S\mid X) & (S\mid S) & 0\\ (S'\mid X) & 0 & (S'\mid S')\end{array}\right|\\
&=&(S'\mid S')-\frac{(S\mid S)(X\mid S')^2}{(S\mid S)(X\mid X)-(X\mid S)^2}.\end{eqnarray*}

For a fixed subspace $\langle \mathbf{x},s\rangle$ let $C=\pxv S'$ denote the Lie projection of the vector $S'$   onto the subspace $\langle \mathbf{X},S\rangle$, and 
$\mathbf{L}=\pxm \langle C\rangle^\perp$ the Lie projection of the orthogonal subspace to $\langle C\rangle$ onto $\langle\mathbf{X},S\rangle$.  

\begin{proposition}\label{ciklivsopu}
If $\Delta(\mathbf{X},S)<0$,  
\begin{enumerate}
\item $ \delta_{S'}(\mathbf{x},s)=(S'-C\mid S'-C)=(S'\mid S')-(C\mid C)$
\item $c\in\langle s\rangle^\perp$ and $\mathbf{l}\subset \langle s'\rangle^\perp$
\item $\delta_{S'}(\ell,s)=(S'\mid S')$ for any $\ell \in \mathbf{l}$.
\end{enumerate}
\end{proposition}
\begin{proof}
\begin{enumerate}
       \item By proposition \ref{sprbaze}, 
       $$\delta_{S'}(\mathbf{x},s)=\frac{\Delta(\mathbf{X},S,S')}{\Delta(\mathbf{X},S)}=
        \Delta(P_{\langle\mathbf{X},S\rangle^\perp}S')=\Delta(S'-\pxv (S'))=(S'-C\mid
S'-C)=(S'\mid S')-(C\mid C).$$
The last equality follows from proposition \ref{projektorji}, since  $(C\mid C)=(C\mid S')$.
	\item $(C\mid S)=(\pxv S'\mid S)=
		(S'\mid \pxv S)=(S'\mid S)=0$. For any $L=\pxv Y$ where $Y\in \langle C\rangle^\perp$
		$(L\mid S')=(\pxv Y\mid S')=
		(Y\mid \pxv S')=(Y'\mid C)=0.$
	\item By the definition, $\Delta(L,S,S')=(S'\mid S')\Delta(L,S)$.
\end{enumerate}
\end{proof}

\begin{theorem} \label{ekstremi1}
Let $\langle\mathbf{x},s\rangle\cap\Omega$ be a hyperbolic $s$-family and $(S\mid S)<0$.
The function 
$$h(x):=\delta_{S'}(x,s)$$
is defined on all $\langle\mathbf{x},s\rangle\cap\Omega$ and achieves its
extreme values on the compact set $\langle\mathbf{x},s\rangle$ at $c$ and on the subspace 
$\mathbf{l}$.
\end{theorem}

\begin{proof} The function $h$ is constant on projective lines through $s$,  so $h(x)=h(y)$ for any $x\in\langle\mathbf{x},s\rangle\cap\Omega$ and $y\in
\left<x,s\right>$.  
According to theorem \ref{discr-family},
$$ \langle\mathbf{x},s\rangle=\langle\mathbf{x},s\rangle\cap\langle
                s \rangle^\perp\oplus\langle s\rangle$$ 
                so it suffices to consider values $h(x)$ for $x\in\langle s\rangle^\perp$. Then $(X\mid X)\neq 0$ and 
$$h(x)= (S'\mid S')-\frac{(X\mid S')^2}{(X\mid X)}$$
is well defined. The differential
$$
dh=-\frac{2(X\mid S')(dX\mid S')(X\mid X)-2(dX\mid X)(X\mid S')^2}{(X\mid X)^2}=
-2(X\mid S)\frac{\left(dX\mid (X\mid X)S'-(X\mid S')X\right)}{(X\mid X)^2}.$$ 
is equal to $0$ in two cases: if 
$(X\mid S')=0$ so $x\in\mathbf{l}$, and $h(x)=(S'\mid S')$ and if   
$$\left(dX\mid (X\mid X)S'-(X\mid S')X\right)=0$$ 
for all $dX\in\langle X,S\rangle$.
In this case $dX=P_{\langle X,S\rangle}dX$,  
$$\left(P_{\langle \mathbf{X},S\rangle}dX\mid (X\mid X)S'-(X\mid S')X\right)=
\left(dX\mid (X\mid X)\pxv S'-(X\mid \pxv S')Y\right)$$
$$=
\left(dX\mid (X\mid X)C-(X\mid C)X\right)=0$$
which implies that $X$ is a multiple of $C$.
\end{proof}

\subsection{Subcycles and cones}
Let us take a closer look at the two cases where $(s,S')$ is equal to $(r,W)$ and $(w,R)$. 

\begin{figure}[h!]
\includegraphics[scale=0.8]{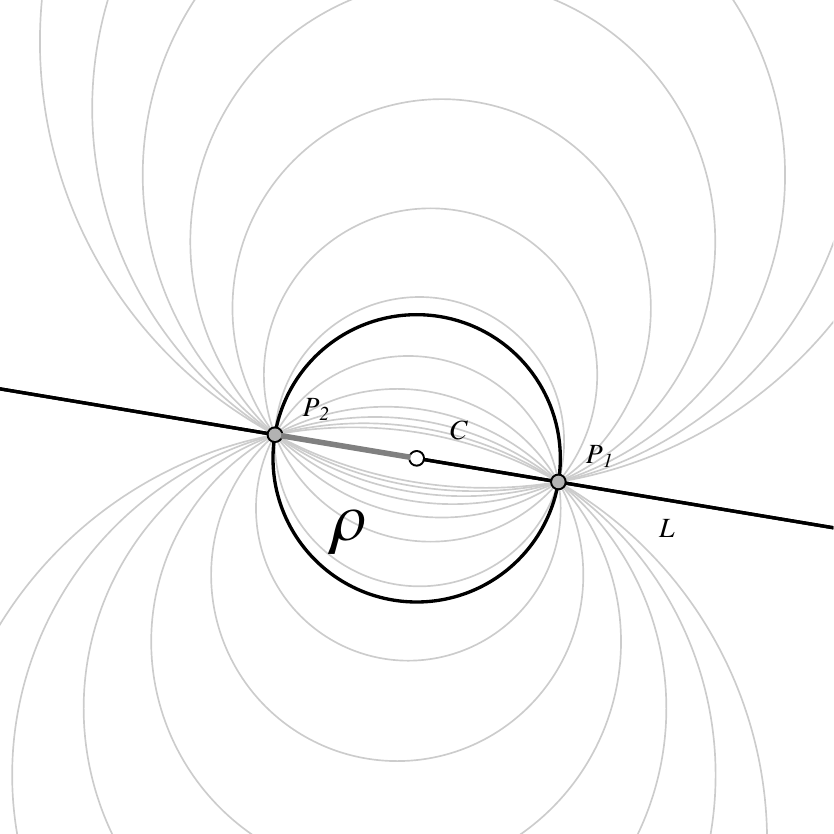}\hfill
\includegraphics[scale=0.8]{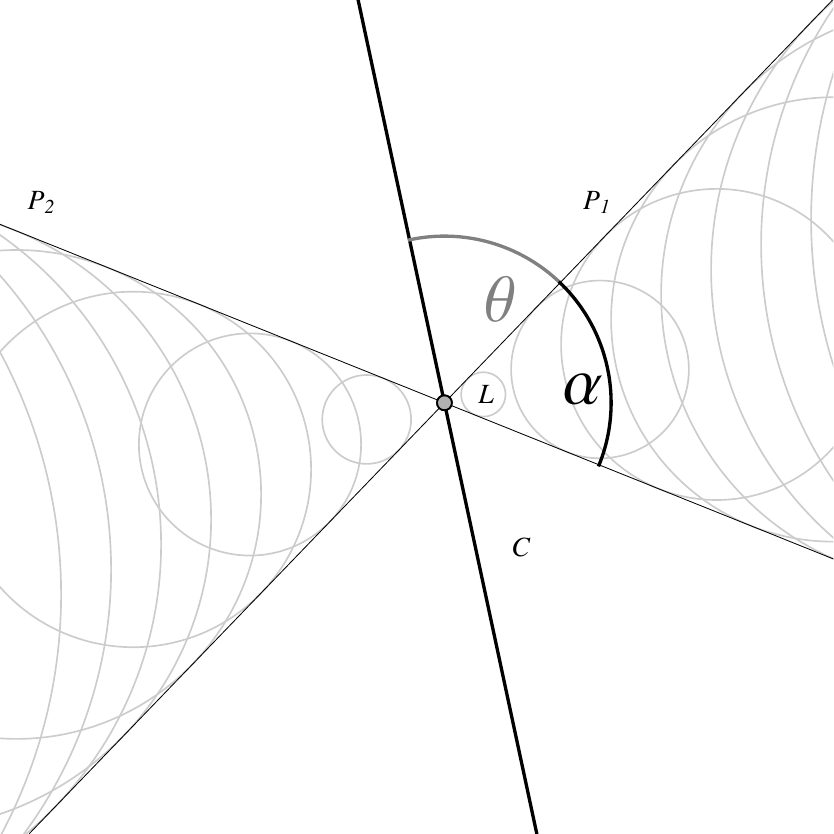}
\caption{In an $r$-family (left) the cycle $c$ projected onto $\Omega$ along $r$ represents the two oriented spheres with the smallest radius, and the subspace 
$\mathbf{l}$ projected onto $\Omega$ along $r$ contains the planes of the family. In a cone family (right) $c$ represents the plane orthogonal to the axis of the cone and $\mathbf{l}$ contains the point cycles of the cone, in particular in $\RR^2$, it is the vertex .}\label{steiner-cone}
\end{figure}

First let $\langle\mathbf{x},r\rangle\cap\Omega$ 
be a hyperbolic $r$-family, that is, a subcycle and let $S'=W$. 
Let $x$ be a cycle from the family.
Then 
$$h(x)=\delta_W(x,r)=\frac{\Delta(X,R,W)}{\Delta(X,R)} = -\frac{1}{\rho^2},$$
where $\rho$ is the radius of the 
corresponding subcycle. Its $W$-discriminant  equals
$$\delta_W(\mathbf{x},r)=-(C\mid C), \quad C=P_{\langle\mathbf{X},R\rangle}W.$$
If $c$ is a non proper cycle and $\delta_W(\mathbf{x},r)\neq 0$, it can be projected onto $\Omega$ along $r$.
We will give a geometric meaning to the discriminant $\delta_W(\mathbf{x},r)$, 
the cycle $c$ and 
the subspace $\mathbf{l}$. 

Assume first that $\delta_W(\mathbf{x},r)\neq 0$.  
The cycle $c=P_{\langle\mathbf{x},r\rangle}(w)$ is a nonproper cycle and according to proposition \ref{ciklivsopu}, belongs to the 
subspace $\langle r \rangle^\perp$. According to theorem \ref{discr-family} it can be projected onto $\Omega$ along $r$, and the projection consists of 
two cycles $c_{1,2}=\langle c, s\rangle\cap\Omega$ which represent one unoriented sphere with both orientations. 
By proposition \ref{ciklivsopu}
\begin{equation}
	\delta_W(\mathbf{x},r)=
	-\frac{(C\mid W)^2}{(C\mid C)}=\frac{\Delta(C_i,R,W)}{\Delta(C_i,R)}=-\frac{1}{\rho^2},
\end{equation} 
where $\rho$ is the radius of the cycles $c_{1,2}$ (see
\ref{steiner-cone} on the left). 
By theorem \ref{ekstremi1}, the spheres $c_{1,2}$ have
the smallest absolute radius in the family.   
Also cycles $\ell\in \mathbf{l} $  can be projected onto the
Lie quadric along $r$. Since $(L\mid W)=0$ holds for all $L\in \mathbf{L} $, the projected cycles represent the planes of the
family. 
The family has no point cycles,
$(\langle \mathbf{x},r\rangle\cap\Omega) \cap\langle
r\rangle^\perp=(\langle \mathbf{x},r\rangle\cap\Omega) \cap(\langle\mathbf{x},r\rangle^\perp\cap\Omega)=\emptyset$.

If $\delta_W(\mathbf{x},r)=0$, the family consists of
planes. In this case the projection $c$ does not exists.  

\subsubsection{\bf{Cones: $s=w$ and $s'=r$}}\label{wr}
The cycles of a hyperbolic $w$-family $\langle \mathbf{x},w \rangle\cap\Omega$ correspond to a family of geometric cycles with common tangent planes. 
The points of tangency form a cone in $\RR^n$, so we will call such a 
family simply a {\em cone}. 
The $R$-discriminant $\delta_R(\mathbf{X},W)$ of such a family is always nonzero since, if 
$\Delta(\mathbf{X},W,R)=0$ 
also $\Delta(\mathbf{X},W)=0$, 
which contradicts the hyperbolicity of the family.  
Using the fact that $(C\mid C)=(R\mid C)$, it is equal to  
\begin{equation}
        \delta_R(\mathbf{x},w)=(P_{\langle\mathbf{X},W\rangle^\perp}R\mid 
		P_{\langle\mathbf{X},W\rangle^\perp}R)=(R-C\mid R-C)=-1-(C\mid C),
\end{equation}  
The cycle
$c$ cannot be projected onto $\Omega$ along $w$ since $(C\mid W)=0$. Its  
projection along $r$ is an unoriented plane with both orientations,  $c_{1,2}$. 
Using coordinates in the chart $\mathcal{U}_r$ we write $C=(\mu,\mathbf{n},0,-\rho)$.  From $(C\mid C)=(R\mid C)=\rho$, we obtain 
$$  
 \|\mathbf{n}\|^2-\rho^2=\rho,\quad
	\|\mathbf{n}\|^2=\rho^2+\rho
$$
Let $C_0=(\mu_0,\mathbf{n}_0,0,\rho_0)=C/\|\mathbf{n}\|$. Then 
$$
	\rho_0=\pm\sqrt{\frac{\rho}{1+\rho}}=\pm\cos\theta
$$
so $|\rho_0|\leq1$ and $\rho\ge 0$.
For an arbitrary cycle $p$ from the cofamily,
\begin{equation*}
        P=\varphi_R(p)=(\mu,\mathbf{n},0,1),\quad 
        (C_0\mid P)=\mathbf{n}_0\cdot\mathbf{n}-\rho_0= 0,\quad
	\mathbf{n}_0\cdot\mathbf{n}=
        \pm\cos\theta
\end{equation*}
The planes forming the cofamily intersect the two planes
$c_{1,2}$ under the angle $\theta$ (see figure \ref{steiner-cone}),
which is complementary to $\alpha/2$ where $\alpha$ is the vertex angle  so  
$$
	\cos^2\theta=\sin^2\frac{\alpha}{2}=\frac{\rho}{\rho+1}\quad\mbox{and}\quad \rho=\tan^2\frac{\alpha}{2}.
$$
The discriminant is then
\begin{equation}\label{cosa}
	\delta_R(\mathbf{x},w)=-1-(C\mid C)=-1-\rho=-1-\tan^2\frac{\alpha}{2}=
	-\frac{1}{\cos^2\frac{\alpha}{2}}.
\end{equation}

All points of the subspace $\mathbf{l}$ can be projected onto $\Omega$ along $w$ and the 
projections represent the point cycles (see \ref{ciklivsopu}). 
In the case of a cone pencil, the subspace $\mathbf{l}$ has
only one cycle representing the vertex of the cone. 
In the case of a cone family generated  by three independent
cycles the cofamily has two elements which represent two tangent planes. The proper cycles in $\mathbf{l}$
are points on the line of intersection of these two planes.  

The family has no planes since 
$$(\langle\mathbf{x},w\rangle\cap\Omega)\cap\langle
w\rangle^\perp\subset (\langle\mathbf{x},w\rangle\cap\Omega)\cap(\langle\mathbf{x},w\rangle^\perp\cap\Omega)=\emptyset.$$  

\section{A family and a cycle}\label{mesan} 
In this section we focus on the relationship between a cycle and an
$s$-family.  Throughout this section we assume that $(S\mid S)\leq 0$, 
$\mathbf{x}=(x_1,\dots,x_k)$, $2\le k\le n-1$,  is a list of proper cycles such that 
$\langle\mathbf{x},s\rangle\cap\Omega$ is a hyperbolic $s$-family, and 
$y\notin\langle\mathbf{x},s\rangle$ is a proper cycle with $(Y\mid S)\neq 0$. 
We define the discriminant of the triple $(\mathbf{x},y,s)$ and describe it in the case $s=r$ and $s=w$ in 
terms of the special determinants of section \ref{simplex}. We next consider the discriminant of a triple $\langle x,y,s\rangle$ with $x\in\mathbf{x}$ as a function $h(x)$ and 
study its critical points.  

\begin{definition}\label{deltas}  The {\em discriminant} of the triple $(\mathbf{x},y,s)$ is the quotient  
  $$\delta(\mathbf{x},y,s) =
  \frac{\Delta(\mathbf{X},Y,S)}{\Delta(\mathbf{X},S)\Delta(Y,S)}.$$
\end{definition}

\begin{theorem} \label{ssopcikel} The
  cofamily $\langle\mathbf{x},s\rangle^\perp\cap\Omega$ contains cycles tangent to $y$ if
  and only if $\delta(\mathbf{x},y,s)\leq 0$. 
\end{theorem}
\begin{proof}
Since $\Delta(\mathbf{X},S)<0$ and
$\Delta(Y,S)=(S\mid S)-(Y\mid S)^2<0$ it follows that  
$\delta(\mathbf{x},y,s)\leq 0$ precisely when
$\Delta(\mathbf{X},Y,S)\leq 0$. By  theorem \ref{proj-family} this is equivalent to 
$\langle\mathbf{x},y,s\rangle^\perp\cap\Omega\neq \emptyset$, which is true if and only if   
$\langle\mathbf{x},s\rangle^\perp\cap\Omega$ intersects $\langle y\rangle^\perp$. But this means that there exists a $Z\in \langle\mathbf{X},S\rangle^\perp$ 
such that $(Z\mid Z)=0$ and $(Z\mid Y)=0$, so $z$ is a proper cycle in $\langle\mathbf{x},s\rangle^\perp$ which is tangent to $y$.
\end{proof}

Assume that $y\notin\langle s\rangle^\perp$. Since $y$ is proper  $\Delta(Y,S)=(Y\mid S)\neq 0$, and the discriminant 
$\delta(x,y,s)$ determines a function 
$$h\:\PP^{n+1}\setminus \{x\mid \Delta(X,S)=0\}\to \RR, \quad h(x)=\delta(x,y,s).$$
The function $h$ is constant on projective lines $\left<x,s\right>$. 
If $(S\mid S)<0$ then by theorem \ref{discr-family} any $x\in\langle\mathbf{x},s\rangle$ has two projections onto $\Omega$ along $s$ and $\Delta(X,S)< 0$. So $h(x)$ is defined for all 
$x\in\langle\mathbf{x},s\rangle$, and $h$ achieves its maximum and minimum on the compact set $\langle\mathbf{x},s\rangle\cap\Omega$. 
If $(S\mid S)=0$
the situation is different, since $\Delta(X,S)=0$ for all $x\in\langle s\rangle^\perp\cap\langle x,s\rangle\neq \emptyset$ and so $h(x)$ is not
bounded on $\langle \mathbf{x},s\rangle\cap\Omega$.

\begin{theorem} \label{osopihinciklih}
The discriminant $\delta(\mathbf{x},y,s)$ is equal to the value of $h$ at the Lie projection $\pxm y$ of $y$ onto $\langle\mathbf{x},s\rangle$. The cycle $\pxm y$ is a critical point of $h$. 
If $(S\mid S)<0$, then $h(x)$ has a second critical point at a cycle
$x_0\in\langle\mathbf{x},s\rangle$ where its value is  ${\displaystyle h(x_0)=\frac{1}{(S\mid S)}}$.
\end{theorem} 

\begin{proof}
Let $(Y\mid S)=1$, that is, $Y=\varphi_S(y)$. Then  
$$h(x)=\frac{(X\mid X)+\left(-2(X\mid S)+(S\mid S)(X\mid Y)\right)(X\mid Y)}{-(X\mid S)^2+(X\mid X)(S\mid S)}.$$ 

  By propositions \ref{sprbaze}
  and \ref{projektorji},
  $$\delta(\mathbf{x},y,s)=-\frac{\Delta(\mathbf{X},Y,S)}{\Delta(\mathbf{X},S)}=
	-\Delta(P_{\langle\mathbf{X},S\rangle^\perp}Y)=
  -\Delta(Y-\pxv Y)=(\pxv Y\mid Y).$$
  On the
  other hand,
  $$\Delta(\pxv Y,Y,S)=\Delta(\pxv Y,S)
	\Delta(P_{\langle\mathbf{X},S\rangle^\perp}Y)=
-(\pxv Y\mid Y)\Delta(\pxv Y,S),$$
  so 
  $$h(\pxm(y))=-\frac{\Delta(\pxv Y,Y,S)}
	{\Delta(\pxv Y,S)}
  =(\pxv Y\mid Y)=\delta(\mathbf{x},s,y)$$
  and the first statement is true.

 Now let $x\in\langle \mathbf{x}, s\rangle$ be different 
from $s$ and let $(X\mid S)=1$. Since $(Y\mid S)=1$, also 
$$(P_{\langle X,S\rangle}Y\mid S)=(Y\mid P_{\langle X,S\rangle}S)=(Y\mid S)=1.$$
  We will consider the cases when $(S\mid S)<0$ and $(S\mid S)=0$ separately.  
  \begin{itemize}	
	\item In the case $(S\mid S)=0$, 
	the function $h(x)$ is defined  for 
	$x\in\langle\mathbf{x},s\rangle\setminus \langle s \rangle^\perp$ and  
	$$h(x)=(X\mid X)-2(X\mid Y),$$
	so
	$$dh=2(dX\mid X)-2(dX\mid Y)=2(dX\mid X-Y).$$
	For $X=\pxv Y$, 	
	$$(dX\mid X-Y)=(dX\mid \pxv Y-Y)=-(dX\mid \pxv ^\perp Y)=0$$
	for all $dX\in\langle\mathbf{X},S\rangle$.
	\item In the second case we may assume that $(S\mid S)=-1$. Since $h(x)$
	is constant on projective 
	lines and since every projective line  $\langle x, s\rangle$  intersects $\langle s\rangle^\perp$ we may pick $x\in
	\langle\mathbf{x},s\rangle\cap\langle
	s\rangle^\perp$. Then,  
	$$h(x)=-1+\frac{(X\mid Y)^2}{(X\mid X)}.$$
	and 
	$$
	  dh=-\frac{2(X\mid Y)}{(X\mid X)^2}\left((X\mid X)(dX\mid Y)-(X\mid Y)(dX\mid X)\right)
	$$ 
	equals $0$ if either $(X\mid Y)=0$
	or $\left(dX\mid(X\mid X)Y-(X\mid Y)X\right)=0$. 
	
	Let $X=\pxv Y$. Then 
		$$
		\left(dX\mid(X\mid X)Y-(X\mid Y)X\right)=
		(\pxv Y\mid Y)\left(dX\mid Y-\pxv Y\right)=0
	$$
	for all $dX\in\langle\mathbf{X},S\rangle$, since
	$Y-\pxv Y=\pxv ^\perp Y$, so $X$ is a critical point of $h$. 
	
	The second critical point is at the cycle $x_0$ where $(X_0\mid Y)=0$ and $h(x_0)=-1$.
 \end{itemize}
\end{proof}

Let us consider the two special cases $S=R$ and $S=W$.
\begin{enumerate}
\item The case of $S=R$. Without loss of generality we
can assume that all cycles $\mathbf{x}$ represent spheres. Then
	\begin{eqnarray*}\delta(\mathbf{x},y,r)&=&
	\frac{\Delta(\varphi_R(\mathbf{x}),\varphi_R(y),R)}
	{\Delta(\varphi_R(\mathbf{x}),R)\Delta(\varphi_R(y),R)}=-
		\frac{\Delta(\varphi_W(\mathbf{x}),\varphi_W(y),R)}
        {\Delta(\varphi_W(\mathbf{x}),R)(\varphi_W(y)\mid R)^2}=\\ &=&
	-\frac{(k+1)^2\vol(q_1,\ldots,q_k,q_y,p)}{\vol(q_1,\ldots,q_k,p)^2\rho^2}=-\frac{v^2}{\rho^2}=-\sin^2\alpha,\end{eqnarray*}
        where $q_i$ and $q_y$ are centers of the spheres 
$x_i$, $i=1,\ldots,k$, and $y$, respectively, 
$p$ is a point in the intersection,
        $\rho$ is the radius of the cycle $y$ and $v$ 
is the height of the simplex spanned by 
$(q_1,\ldots,q_k,q_y,p)$ to the base simplex spanned by 
        $(q_1,\ldots,q_k,p)$, and 
        $\alpha$ is the angle between the plane of the base
        simplex and the line segment
        which connects the center $q_y$ with an intersection point $p$. 
        
        \item The case of $S=W$. In this case
        $$\delta(\mathbf{x},y,w)=
	\frac{\Delta(\varphi_W\mathbf{x},\varphi_W(y),W)}{
	\Delta(\varphi_W(\mathbf{x}),W)\Delta(\varphi_W(y),W)}
        =-\frac{k^2\vol(q_1,\ldots,q_k,q_y)^2}{\vol(q_1,\ldots,q_k)^2}=-v^2,
	$$
	where the points $(q_1,\ldots,q_k,q_y)$ span the contact simplex that lies in the common
	tangent plane of the cycles $(\mathbf{x},y)$,  
	$(q_1,\ldots,q_k)$ span the contact simplex of $\mathbf{x}$ (see \eqref{Vol}) and 
	$v$ is the height
	of the first simplex above to the second. The height $v$ is 
	the tangential distance from the cycle $y$ to
	the cone.

\end{enumerate}

\section{Two families}\label{dvaenaka}
In this section we consider the relation between two $s$-families of the same dimension. We define the discriminant of the triple $(\mathbf{x},\mathbf{y},s)$ and show that 
it can be expressed in terms of the eigenvalues of the product $\pxv \pyv $. We discuss the critical points of the function obtained by considering the determinant of triples $(x,y,s)$, 
$x\in \langle\mathbf{x},s\rangle$, $y\in\langle\mathbf{y},s\rangle$ as a function of $(x,y)$ and show that they are closely associated to the fixed points of the products $\pxm\pym$ and $\pym\pxm$ (which correspond to the eigenvectors of the products 
$\pxv \pyv $ and $\pyv \pxv $). We also give a geometric interpretation in the two cases $s=r$ and $s=w$.  

\begin{definition}
The {\em discriminant} of two hyperbolic $s$-families $\langle\mathbf{x},s\rangle\cap\Omega$ and $\langle\mathbf{y},s\rangle\cap\Omega$
spanned by linearly independent proper cycles $\mathbf{x}=(x_1,\dots,x_{k})$ and $\mathbf{y}=(y_1,\dots,y_{k})$ is defined by
\begin{equation}
\delta(\mathbf{x},\mathbf{y},s)=
\frac{\Delta(\mathbf{X},\mathbf{Y},S)}{\Delta(\mathbf{X},S)\Delta(\mathbf{Y},S)}.
\end{equation}
\end{definition}

Because of the hyperbolicity of the two families and the linear independence of the spanning cycles, the family 
$\langle\mathbf{x},\mathbf{y},s\rangle\cap\Omega$ 
is hyperbolic if and only if $\delta(\mathbf{x},\mathbf{y},s)<0$.  

\begin{proposition}\label{dvasopa1}
The value $\delta(\mathbf{x},\mathbf{y},s)$ of the discriminant of
two hyperbolic $s$-families $\langle\mathbf{x},s\rangle\cap\Omega$ and $\langle\mathbf{y},s\rangle\cap\Omega$ is equal
to $0$ if either they contain a common cycle 
$z\in \langle\mathbf{x},s\rangle\cap\langle\mathbf{y},s\rangle$ 
or the corresponding cofamilies have a nonempty intersection.
\end{proposition}
\begin{proof}
In the first case, the intersection $\left<\mathbf{X},S\right>\cap
\left<\mathbf{Y},S\right>$ contains a vector $Z\notin\langle S\rangle$, so the
columns of $\Delta(\mathbf{X},\mathbf{Y},S)$ are linearly dependent.  In the
second case, the equation $(Z\mid Z)=0$ has a nontrivial solution in
the subspace $\langle\mathbf{X},\mathbf{Y},S\rangle$. The resulting homogeneous
system
$$
(Z\mid S)=0,\quad (Z\mid X_i)=0,\quad (Z\mid Y_i)=0,\quad i=1,\dots,k
$$ 
thus has a nontrivial solution and the coefficient matrix of
this system, which equals $\Delta(\mathbf{X},\mathbf{Y},S)$, must be $0$.
\end{proof}

\subsection{Eigenvalues and eigenvectors of a product of two projectors}
We will need some properties of the eigenvectors and eigenvalues of the products $\pxv \pyv$ and $\pyv \pxv$ which we prove in this section. 
Since $S\in \langle\mathbf{X},S\rangle\cap\langle\mathbf{Y},S\rangle$, one of the eigenvalues is $1$, and one of the corresponding eigenvectors is $S$. 
The eigenvectors corresponding to the eigenvalue $1$ are precisely the vectors in the intersection $\langle\mathbf{X},S\rangle\cap\langle\mathbf{Y},S\rangle$. 

\begin{proposition}\label{lastnevrednosti}
The products of the Lie projectors $\pxv \pyv $ and
  $\pyv \pxv $ have the same
  eigenvalues.  If $E$ is an eigenvector of $\pxv \pyv$ then $F=\pyv E$ is 
an eigenvector of $\pyv \pxv$. 
  Eigenvectors which belong to different nonzero eigenvalues are Lie orthogonal.
\end{proposition}
\begin{proof}
Multiplying both sides of the equation 
	$\pxv \pyv E=\lambda
	E$ by
	$\pyv $, we get\\
	$$\pyv 
	\pxv \pyv E=\lambda
        \pyv E.$$ If 
	$\pyv E=F$, it follows that   
	$\pyv \pxv F=\lambda
        F$. So the vector $F$ is an eigenvector of
	$\pyv \pxv $
	corresponding to the same eigenvalue $\lambda$.	
	
   Let $E_i$ and $E_j$ be eigenvectors of 
	$\pxv \pyv $,
	corresponding to eigenvalues 
	$\lambda_i$ and $\lambda_j$ respectively, where 
	$\lambda_i\ne \lambda_j$
	and $\lambda_i,\lambda_j\ne0$. Then
	\begin{equation*}
	(\pxv \pyv E_i\mid 
	\pxv \pyv E_j)=
	\lambda_i\lambda_j(E_i\mid E_j).
	\end{equation*}
	On the other hand
	\begin{eqnarray*}
	(\pxv \pyv E_i\mid 
        \pxv \pyv E_j)&=& 
	(\pyv 
        \pxv \pyv E_i\mid E_j)\\
	&=&(\pyv 
        \pxv \pyv 
	E_i\mid \pxv E_j)\\
	&=&
	(\pxv \pyv 
        \pxv \pyv E_i\mid E_j)\\
        &=&
	\lambda_i^2(E_i\mid E_j). 
	\end{eqnarray*}
	Since $\lambda_i\ne\lambda_j$ it follows that
	$(E_i\mid E_j)=0.$ 
\end{proof}

From now on, let $\mathbf{E}$ and $\mathbf{F}$ denote lists of linearly independent eigenvectors of $\pxv \pyv$ and $\pyv \pxv$, respectively,  which are independent from $S$.

\begin{proposition}
  Assume that $\delta(\mathbf{x},\mathbf{y},s)<0$. 
	 \begin{enumerate} 
  \item The eigenvalues of the products $\pxv    \pyv    $ and $\pyv \pxv$ different from zero are positive.
	The restriction of the Lie form to $\langle\mathbf{E}\rangle$ and $\langle\mathbf{F}\rangle$
	is positive definite. 
  \item If $E\in\mathbf{E}$ and $F\in\mathbf{F}$ are eigenvectors belonging to the same eigenvalue such that $(E\mid E)=1$
	and  $(F\mid F)=1$, then
	$(E\mid F)=\sqrt{\lambda}$
	where  $\lambda$ is the corresponding eigenvalue.
	
  \item  If $(S\mid S)<0$, 
	then all eigenvalues are nondegenerate, and 
	the corresponding eigenvectors span
	the subspaces $\langle\mathbf{X},S\rangle$ and 
	$\langle\mathbf{Y},S\rangle$. All cycles 
	$e\in\langle\mathbf{e}\rangle$ and $f\in\langle\mathbf{f}\rangle$ different from $s$, can be  
        projected onto the Lie quadric along $s$.\\
	If $(S\mid S)=0$ the eigenvalue 1 is degenerate and the
	cycles, corresponding to the eigenvectors independent from $S$,  
	cannot be projected onto the Lie quadric along $s$.   
  \end{enumerate}
\end{proposition}
\begin{proof}
\begin{enumerate}
\item 	Since $\delta(\mathbf{x},\mathbf{y},s)<0$, $S$ is the only eigenvector corresponding to $1$ and by proposition \ref{lastnevrednosti} 
	$\langle\mathbf{E}\rangle\subset S^\perp$ and also $\langle\mathbf{F}\rangle\subset S^\perp$. By theorem 
	\ref{discr-family} this implies that $(E\mid E)> 0$ and $(F\mid F)> 0$ for all $E\in \mathbf{E}$ and $F\in\mathbf{F}$, and the restriction of the Lie form to the span
	$\langle\mathbf{E}\rangle$ and $\langle\mathbf{F}\rangle$ 
	of eigenvectors is positive definite. 
	Also,  
	\begin{eqnarray*}
	(F\mid F)&=&(\pyv    E\mid \pyv    E)=(\pyv    E\mid E)\\
	&=&(\pyv    E\mid \pxv    E)=(\pxv    \pyv    E\mid E)\\
	&=&\lambda(E\mid E)>0,\end{eqnarray*}
	so $\lambda>0$.
\item If $E$ and $F$ are eigenvectors belonging to the same eigenvalue $\lambda$ and $(E\mid E)=1$ and $(F\mid F)=1$ 
then $\pyv    E=c F$, $(c F\mid c F)=c^2$. On the other hand 
\begin{eqnarray*}
(c F\mid c F)&=&(\pyv    E\mid \pyv    E)=(\pyv    E\mid E)\\
&=&(\pyv    E\mid \pxv    E)=(\pxv    \pyv    E\mid E)=\lambda,\end{eqnarray*}
so $c=\sqrt{\lambda}$. 
Also,
$$(E\mid F)=(E\mid \frac{1}{\sqrt{\lambda}}\pyv    E)=(\pxv    E\mid \frac{1}{\sqrt{\lambda}}\pyv    E)=
(E\mid \frac{1}{\sqrt{\lambda}}\pxv    \pyv    E)=\sqrt{\lambda}$$
\item This follows directly from theorem \ref{discr-family}. 
\end{enumerate}
\end{proof}
\subsection{The discriminant}
Let $\langle\mathbf{x},s\rangle\cap\Omega$ and $\langle\mathbf{y},s\rangle\cap\Omega$ be two
$k$-dimensional hyperbolic families with 
$\delta(\mathbf{x},\mathbf{y},s)<0$. 
We discuss the cases $(S\mid S)<0$ and $(S\mid S)=0$ separately. 
\subsubsection{The case of $(S\mid S)<0$}\label{steiner-base}
In this case $(\mathbf{E},S)$ and  $(\mathbf{F},S)$ are bases of $\langle\mathbf{X},S\rangle$ and  $\langle\mathbf{Y},S\rangle$ consisting of eigenvectors of 
the products $\pxv\pyv$ and $\pyv\pxv$, respectively, such that 
$(E_i\mid E_j)=\delta_{ij}$, $(F_i\mid F_j)=\delta_{ij}$, $F_i=\pxv E_i$ and 
$(E_i\mid F_j)=\sqrt{\lambda_i}\delta_{ij}$ for all $i,j=1,\ldots,k$. 
Each cycle $e_i\in\mathbf{e}$ and $f_j\in\mathbf{f}$ can be projected onto the Lie quadric in the
direction of $s$. The discriminant equals
$$
	\delta(\mathbf{x},\mathbf{y},s)=
	  \left|\begin{array}{ccccccc}
          1 & 0 &\cdots & \sqrt{\lambda_1} & 0 & 0 \\
          0 & 1 &\cdots & 0 & \sqrt{\lambda_2} & 0 \\
	  \hdotsfor{6}\\
          \sqrt{\lambda_1} & 0 &\cdots& 1 & 0 & 0 \\
          0 &  \sqrt{\lambda_2}&\cdots & 0 & 1 & 0\\
          0 & 0 &\dots& 0 & 0 & -1
	\end{array}\right|=
	\left|\begin{array}{ccc}
	  \mathbf{I}&\mathbf{\Lambda}&\mathbf{0}\\
          \mathbf{\Lambda}&\mathbf{I}&\mathbf{0}\\
	  \mathbf{0}^T&\mathbf{0}^T&1  	
	\end{array}\right|=-|\mathbf{I}-\mathbf{\Lambda}^2|=-\prod_{i=1}^k(1-\lambda_i),
$$ 
where 
$$\Lambda=\left[\begin{array}{cccc}
	\sqrt{\lambda_1} & 0& \cdots & 0\\
	0 &\sqrt{\lambda_2} & \cdots & 0\\
	\hdotsfor{4}\\
	0 & 0 & \cdots & \sqrt{\lambda_k}
	\end{array}\right]$$ 
	and $\mathbf{I}$ is the identity matrix
of the size $k\times k$. 

Let $S=R$. Each cycle $e_i\in\mathbf{E}$ and
$f_i\in\mathbf{F}$ belongs to the space $\langle r \rangle^\perp$ and can be
projected onto the Lie quadric in the direction of the $r$. The projections 
$\hat{e}_i$ and $\hat{f}_i$ have homogeneous coordinates 
$\hat{E}_i=E_i-R$ and $\hat{F}_i=F_i-R$ (since $(E_i\mid E_i)=1$ and $(F_i\mid F_i=1)$ and 
$$(\hat{E}_i\mid \hat{F}_i)=(E_i-R\mid F_i-R)=\sqrt{\lambda_i}-1.$$ 
Since $(\hat{E}_i\mid R)=(\hat{F}_i\mid R)=1$ it follows that $\hat{E}_i$ and $\hat{F}_i$ are homogeneous coordinates in the chart $\mathcal{U}_r$ and 
by \eqref{cikli1} 
$\sqrt{\lambda_i}=-\cos\alpha_i$, where $\alpha_i$ is the angle of intersection of the geometric cycles corresponding to $\hat{e}_i\in\langle\mathbf{x},s\rangle\cap\Omega$ 
and $\hat{f}_i\in\langle\mathbf{y},s\rangle\cap\Omega$. The eigenvalue $\lambda_i$ is thus equal to $\cos^2\alpha_i$. 

\subsubsection{The case of $(S\mid S)=0$}
Since $1$ is a degenerate eigenvalue, the eigenvectors $(\mathbf{E},S)$ and $(\mathbf{F},S)$ do not span the subspaces $\langle\mathbf{X},S\rangle$ and 
$\langle\mathbf{Y},S\rangle$, so there exist vectors $T$ and $U$ such that
$\langle T, S\rangle=\langle \mathbf{E} \rangle^\perp\cap 
\langle\mathbf{X,S}\rangle$ and $\langle U,
S\rangle=\langle \mathbf{F} \rangle^\perp\cap 
\langle\mathbf{Y,S}\rangle$. 
Let  $\hat{t}$ and $\hat{u}$ be the projections of $t$ and $u$ onto the quadric along $s$, 
and $\hat{T}$ and $\hat{U}$ homogeneous coordinates in
the chart $\mathcal{U}_s$ so that a basis of the space $\langle T, S\rangle$, $(\hat{T},S)$ and
similarly $(\hat{U},S)$. Then 
$\langle\mathbf{E},\hat{T},S\rangle=\langle \mathbf{X},S\rangle$ and  
$\langle\mathbf{F},\hat{U},S\rangle=\langle \mathbf{Y},S\rangle$ 
and the discriminant is 
 
\begin{equation}\label{zafix}
        \delta(\mathbf{x},\mathbf{y},s)=\left|\begin{array}{ccc}
          \mathbf{I}&\mathbf{\Lambda}&\mathbf{0}\\
          \mathbf{\Lambda}&\mathbf{I}&\mathbf{0}\\
          \mathbf{0}^T&\mathbf{0}^T&D   
        \end{array}\right|,
\end{equation}
where 
$$\Lambda=\left[\begin{array}{cccc}
        \sqrt{\lambda_1} & 0& \cdots & 0\\
        0 &\sqrt{\lambda_2} & \cdots & 0\\
        \hdotsfor{4}\\
        0 & 0 & \cdots & \sqrt{\lambda_{k-1}}
        \end{array}\right],$$
        $\mathbf{I}$ is the identity matrix of 
	size of $k-1\times k-1$ and\\ 
	$$D=\left[\begin{array}{ccc}
		0&(\hat{T}\mid \hat{U})&1\\
		(\hat{T}\mid \hat{U})&0&1\\
		1&1&0		
		\end{array}\right|.$$
		
Thus, 
$$
	\delta(\mathbf{x},\mathbf{y},s)=2(\hat{T}\mid \hat{U})
\prod_{i=1}^{k-1}(1-\lambda_i).
$$
 
Let $S=W$. 
Cycles $e\in\mathbf{e}$ and $f\in\mathbf{f}$ belong to the space 
$\langle w\rangle^\perp$ and cannot be
projected onto
the Lie quadric along $w$. Let  $\hat{e}_i$ and $\hat{f}_i$ be the projections  
along $r$. These cycles represent the planes of the two $w$-families. Local coordinates in the chart $\mathcal{U}_r$ are 
$\hat{E}_i=E_i-R$ and
$\hat{F}_i=F_i-R$, so $(\hat{E}_i\mid \hat{F}_i)=1-\lambda_i=\sin^2\alpha_i$ where $\alpha_i$   
the angle between the planes $\hat{e}_i$ and
$\hat{f}_i$. On the other hand, $\hat{T}$ and $\hat{U}$ are local coordinates in $\mathcal{U}_w$, 
$(\hat{T}\mid\hat{U})=-\frac{d^2}{2}$, where $d$ is the tangential distance between the cycles
$\hat{t}$ and $\hat{u}$. 
As we will see later, this tangential distance is the minimal 
tangential distance between the cycles of the families $\langle\mathbf{x},w\rangle\cap\Omega$ and
$\langle\mathbf{y},w\rangle\cap\Omega$.  
Thus 
$$\delta(\mathbf{x},\mathbf{y},w)=-d^2\prod_{i=1}^{k-1}\sin^2\alpha_i.$$

\subsection{Critical points}
Let $\langle\mathbf{x},s\rangle\cap\Omega$ and $\langle\mathbf{y},s\rangle\cap\Omega$ be
$k$-parametric hyperbolic families, such that
$\delta(\mathbf{x},\mathbf{y},s)<0$. The two products
$\pxv \pyv $ and
$\pyv \pxv $ both have rank
$k+1$ and there are at most $k+1$
nonzero eigenvalues. One of them is $1$, and its corresponding
eigenvector is $S$.

Now consider the function
\begin{equation}
h(x,y)=\frac{\Delta(X,Y,S)}{\Delta(X,S)\Delta(Y,S)}
\end{equation}
where $x\in\langle\mathbf{x},s\rangle$ and $y\in\langle\mathbf{y},s\rangle$ vary independently.
\begin{theorem}\label{dvasopa}
  The function $h(x,y)$ restricted to 
$\langle\mathbf{x},s\rangle\cap\Omega\times\langle\mathbf{y},s\rangle\cap\Omega$
  has at least one critical point.
  If $(e,f)$ is a critical point of $h(x,y)$ then 
  $e$ is a critical point of 
  $h_1(y)=\frac{\Delta(\mathbf{X},Y,S)}{\Delta(\mathbf{X},S)\Delta(Y,S)}$
  and $f$ is a critical point of
  $h_2(x)=\frac{\Delta(X,\mathbf{Y},S)}{\Delta(X,S)\Delta(\mathbf{Y},S)}$.
  \begin{enumerate}
  \item If $(S\mid S)<0$ the function $h(x,y)$
  restricted to
  $\langle\mathbf{x},s\rangle\cap\Omega\times\langle\mathbf{y},s\rangle\cap\Omega$
  has a critical point
  at a pair $(\hat{e},\hat{f})$ obtained by projecting a fixed point $e$ of 
  $\pxm\pym$ and
  $f$ of $\pym(e)$ of
  $\pym\pxm$
  onto $\langle\mathbf{x},s\rangle\cap\Omega$ and
  $\langle\mathbf{y},s\rangle\cap\Omega$, respectively, along $s$. 
  All critical points are of this type.
  \item If $(S\mid S)=0$, there exists a unique fixed projective line $\langle t, s\rangle$ of 
  $\pxm\pym$ and a unique fixed projective line $\langle u, s \rangle$ 
  of $\pym\pxm$. These intersect the Lie quadric in points $\hat{t}$ and $\hat{u}$. The point $(\hat{t},\hat{u})$ is the only  
 critical point of 
  $h(x,y)$ restricted to the $\langle\mathbf{x},s\rangle\cap\Omega\times\langle\mathbf{y},s\rangle\cap\Omega$.
  \end{enumerate}
\end{theorem}
\begin{proof}
  Consider the function 
  	$$h_1(y)=\frac{\Delta(\mathbf{X},Y,S)}{\Delta(\mathbf{X},S)\Delta(Y,S)}=\frac{\Delta(P_{\langle\mathbf{X,S}\rangle^\perp}Y)}{\Delta(Y,S)}=
	\frac{(Y\mid Y)-(Y\mid P_{\langle \mathbf{X},S\rangle}Y)}{\Delta(Y,S)}.$$ 
  \begin{enumerate}
  \item  Let $(S\mid S)<0$ and let $(e,f)$ be such that $e$ is a fixed point of 
	$\pxm\pym$, $f$ is a fixed point of
  	$\pym\pxm$, and
  	$e=\pxm f$ (and thus
  	$f=\pym e$). 
  	According the theorem \ref{osopihinciklih}, $e$ is a critical point of $h(x,f)$ and $f$ is a critical point  
  	of $h(e,y)$.
  	Without loss of generality we
  	may assume that $(S\mid S)=-1$. 
  	Since $h_1$ is constant on projective lines $\left<y,s\right>$, it
  	suffices to consider its values on the subspace
  	$\left<s\right>^\perp$ to find its critical points on
  	$\left<\mathbf{y},s\right>$. Let $y\in
  	\left<\mathbf{y},s\right>\cap\left<s\right>^\perp$, so $\Delta(Y,S)=-(Y\mid Y)$. 
	By theorem \ref{discr-family}, 
  	$\langle\mathbf{y},s\rangle\cap\left<s\right>^\perp\cap\Omega=\emptyset$,
  	and $(Y\mid Y)>0$ and
  	$$h_1(y)=-\frac{\Delta(Y-\pxv Y)}{(Y\mid Y)}= 
        -\left(1-\frac{\pxv Y\mid Y)}{(Y\mid Y)}\right).$$
  	The critical points $y$ satisfies the condition
  	$$
	\frac{2(\pxv Y\mid dY)(Y\mid Y)-
	(\pxv Y\mid Y)(Y\mid dY)}{(Y\mid Y)^2}=0
	$$
  	for all
  	$dY\in\langle\mathbf{Y},S\rangle$. 
	Without loss of generality we can assume that $(Y\mid Y)=1$  so
  	$$\pxv Y=(\pxv Y\mid
  	Y)Y.$$
  	Applying $\pxv $ on both sides we obtain
  	$$\pyv \pxv Y=
	(\pxv Y\mid Y)Y,$$
  	which is true if and only if $Y$ is an eigenvector $F$.  
	So $f$ is a critical point oh $h_1(y)$ and $\hat{f}$ is a critical point of 
	its restriction to 
	$\langle\mathbf{x},s\rangle\cap\Omega$. Clearly, all critical points are of this kind.
  \item Let $(S\mid S)=0$.  In this case the function $h_1(y)$ is defined on 
	$\langle\mathbf{y},s\rangle\setminus\langle s\rangle^\perp$ 
	and constant on projective lines $\langle x,s
	\rangle$, and
	$$
	h_1(y)=	\frac{(Y\mid Y)-(Y\mid P_{\langle\mathbf{X},S\rangle} Y)}{(Y\mid S)^2}.$$
So 
$$dh_1=2\frac{\left((dY\mid Y)-(dY\mid \pxv Y)\right)
(Y\mid S)^2-\left((Y\mid Y)-(Y\mid \pxv Y)\right)
(Y\mid S)(dY\mid S)}{(X\mid S)^4}.$$
	The point $y$ is critical if 
	$$
		\left(dY\mid (Y-\pxv Y)-\frac{S}{(Y\mid S)}(Y\mid Y-\pxv Y)\right)=0,	
	$$
	that is, if 
	 $$Y-\pxv Y=
                \frac{S}{(Y\mid S)}(Y\mid Y-\pxv Y).$$
Without loss of generality we may assume that $(Y\mid S)=1$. After applying $\pyv $ and rearranging, we obtain 
	\begin{equation}\label{tix}
		\pyv \pxv Y=
		Y-S(Y\mid Y-\pxv Y).
	\end{equation}
	The solution $Y$ of the equation is not an
	eigenvector. Let $U\in\langle\mathbf F\rangle^\perp$. The projective line $\langle u, s\rangle$ is a fixed line of the 
	$\pym\pxm$.
        The line  intersects the quadric in the unique
	critical point of $h_1(y)$ restricted to $\langle\mathbf{y},s\rangle\cap\Omega$. 
\end{enumerate}
\end{proof}

\subsection{Two pencils in $\RR^3$}
In this section we will focus on geometric cycles and pencils in
$\RR^3$. 

 Let $\langle\mathbf{x},r\rangle\cap\Omega$ and
$\langle\mathbf{y},r\rangle\cap\Omega$ be two hyperbolic (Steiner) pencils, which
determine two subcycles (circles or lines) in $\RR^3$. 
Now $\langle\mathbf{x}\rangle=\langle x_1,x_2\rangle$ where $x_1$ and
$x_2$ are proper cycles and  similarly for
$\langle\mathbf{y}\rangle=\langle y_1,y_2\rangle$.
Since $(R\mid\! R)=-1$, the critical points
of the the function $h$ on
$\langle\mathbf{x},r\rangle\times\langle
\mathbf{y},r\rangle$
are of the form $(e,f)$, where $e$ and $f=P_{\langle\mathbf{y},r\rangle}(e)$ are
fixed points of the product of projectors.  One of these critical
points is $(r,r)$. We will be interested in the critical points on 
$\langle\mathbf{x},r\rangle\cap\Omega\times\langle\mathbf{y},r\rangle\cap\Omega$.  
Since
$\langle\mathbf{x},r\rangle\cap\Omega\times\langle\mathbf{y},r\rangle\cap\Omega$ is a
compact set, and $h$ is continuous, there must be two 
critical points $(\hat{e}_1,\hat{f}_1)$ and $(\hat{e}_2,\hat{f}_2)$ where the
maximum and minimum are achieved.
\begin{theorem}Let $\langle\mathbf{x},r\rangle\cap\Omega$ and
  $\langle\mathbf{y},r\rangle\cap\Omega$ be independent in the sense that they have
  no common cycles.   The corresponding subcycles in $\RR^3$ intersect if
  $\delta(\mathbf{x},\mathbf{y},r)=0$,
  they are linked if 
  $\delta(\mathbf{x},\mathbf{y},r)<0$, and are unlinked
  if $\delta(\mathbf{x},\mathbf{y},r)>0$.
  If $\delta(\mathbf{x},\mathbf{y},r)<0$, then the smallest and
  largest angle of intersection between a cycle $x\in\langle\mathbf{x},r\rangle\cap\Omega$ and
  a cycle $y\in\langle\mathbf{y},r\rangle\cap\Omega$ are reached at pairs
  $(\hat{e}_1,\hat{f}_1)$
  and $(\hat{e}_2,\hat{f}_2)$, which are obtained by projecting the fixed
  points $e_1,e_2$ of $P_{\langle\mathbf{x},r\rangle}P_{\langle\mathbf{y},r\rangle}$ and
  $f_1=P_{\langle\mathbf{y},r\rangle}(e_1)$, $f_2=P_{\langle\mathbf{y},r\rangle}(e_2)$ of
  $P_{\langle\mathbf{y},r\rangle}P_{\langle\mathbf{x},r\rangle}$ onto $\Omega$ along 
  $r$.
\end{theorem}
\begin{proof} 
    The first statement follows directly from proposition \ref{dvasopa1}.
    Since $\Delta(\mathbf{X},R)<0$ and $\Delta(\mathbf{Y},R)<0$, the
    sign of $\delta(\mathbf{x},\mathbf{y},r)$ is equal to the sign
    of $\Delta(\mathbf{X},\mathbf{Y},S)$. Let $(E_1,E_2,R)$ and
    $(F_1,F_2,R)$ be bases of $\left<\mathbf{X},R\right>$ and
    $\left<\mathbf{Y},R\right>$ consist of eigenvectors of 
    $P_{\left<\mathbf{X},R\right>}P_{\left<\mathbf{Y},R\right>}$ and
    $P_{\left<\mathbf{Y},R\right>}P_{\left<\mathbf{X},R\right>}$, respectively.  
    Since the Lie
    form restricted to $\left<\mathbf{X},R\right>$ or
    $\left<\mathbf{Y},R\right>$ has index $1$ and $(R\mid R)<0$,
    we can normalized them so that$(E_i,E_j)=\delta_{ij}$,
    $(F_i\mid F_j)=\delta_{ij}$ and
	$(E_i\mid F_j)=\delta_{ij}\sqrt{\lambda_i}$, where
    $\lambda_i$ are corresponding eigenvalues (see the section
    \ref{steiner-base}).
    Then,
    $$
     \Delta(\mathbf{X},\mathbf{Y},S)=-(1-\lambda_1)(1-\lambda_2)
    $$
    If $\delta(\mathbf{x},\mathbf{y},r)>0$, then
    not all pairs of the cycles $(x,y)$,
    $x\in \langle\mathbf{x},r\rangle\cap\Omega$, $y\in
\langle\mathbf{y},r\rangle\cap\Omega$ intersects, which implies
    that the subcycles are unlinked.
    If $\delta(\mathbf{x},\mathbf{y},r)<0$ then by theorem
    \ref{dvasopa} the extreme values of $h(x,y)=-\sin^2\alpha$
    $(x,y)\in\langle\mathbf{x},r\rangle\cap\Omega\times\langle\mathbf{y},r\rangle\cap\Omega$ 
    are achieved at the two pairs
    $(\hat{e}_1,\hat{f}_1)$ and $(\hat{e}_2,\hat{f}_2)$ with maximal and
    minimal angles of intersection $\alpha$.
\end{proof}

The situation with cones (and in general with
$s$-pencils, where $(S\mid S)=0$) is different than in the case of
subcycles (where $(S\mid S)<0$). 

\begin{figure}[h!]
	\includegraphics[scale=0.20]{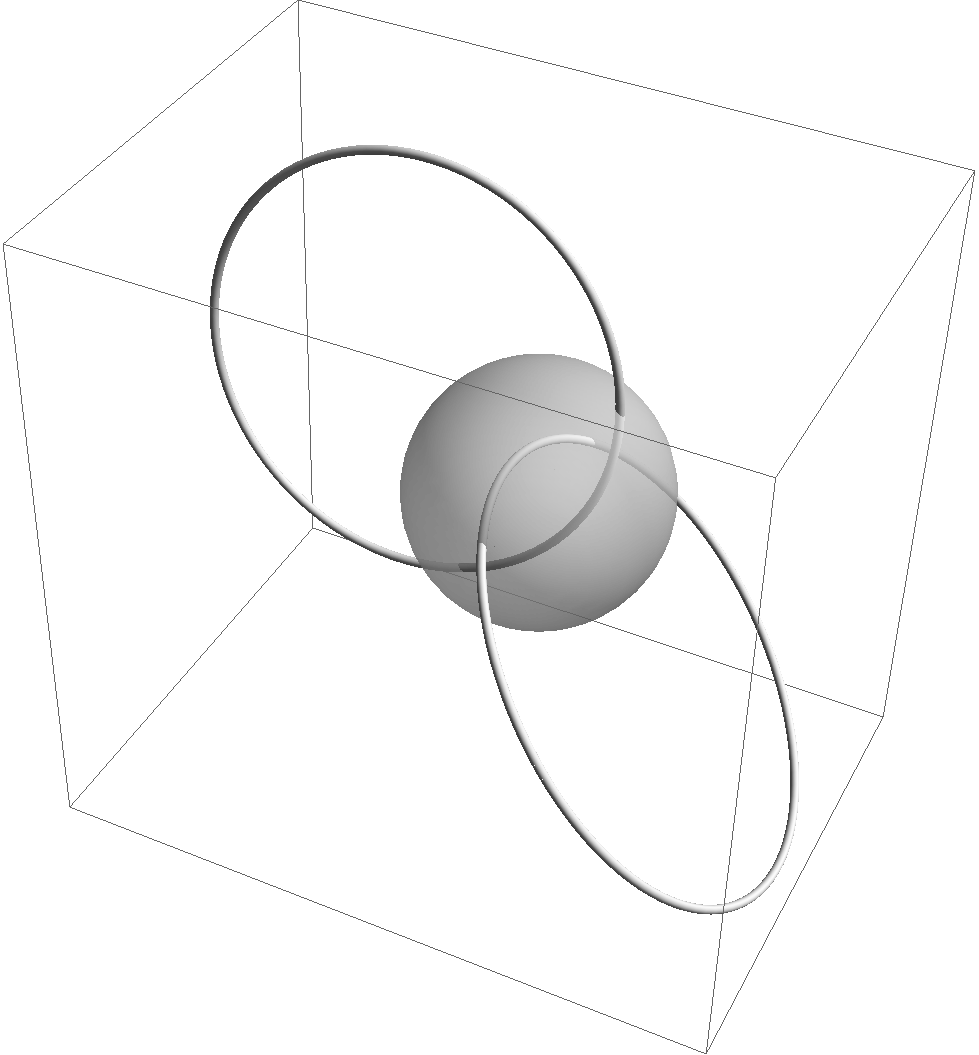}\hfill
        \includegraphics[scale=0.18]{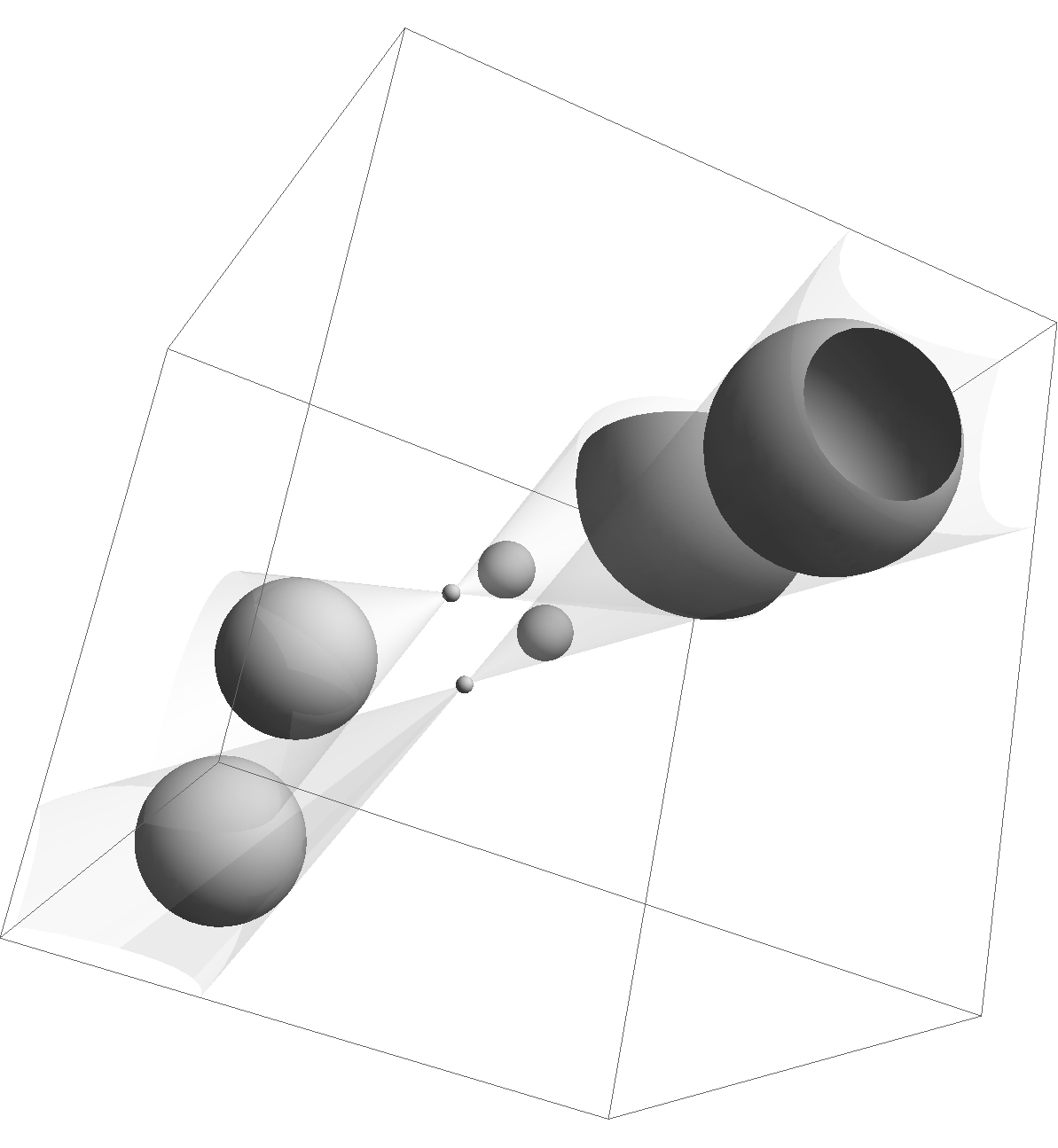} 
\caption{The case of two linked circles and two nonintersecting
cones, $\delta<0$}\label{Linked}
\end{figure}

\begin{theorem}
  Let $\langle\mathbf{x},w\rangle\cap\Omega$ and $\langle\mathbf{y},w\rangle\cap\Omega$ be cones.
\begin{enumerate} 
\item If $\delta(\mathbf{x},\mathbf{y},w)=0$, then two cones
  either share a common cycle or an oriented tangential plane
  or they have parallel axes.  
\item If $\delta(\mathbf{x},\mathbf{y},w)<0$, then for any pair of
  cycles
  $(x,y)\in\langle\mathbf{x},w\rangle\cap\Omega\times\langle\mathbf{y},w\rangle\cap\Omega$ 
  their tangential
  distance exists. The minimal tangential distance is reached at a
  pair $(\hat{e},\hat{f})$ which is obtained by projecting the fixed 
  points $e$ of $P_{\langle\mathbf{x},w\rangle} 
  P_{\langle\mathbf{y},w\rangle}$ 
  and $f$  of 
  $P_{\langle\mathbf{y},w\rangle}P_{\langle\mathbf{x},w\rangle}$ along $w$ onto 
  $\Omega$.
  \item If $\delta(\mathbf{x},\mathbf{y},w)>0$ then there exist two cycles, one
  of each family, which have no tangential distance. 
\end{enumerate}
\end{theorem}
\begin{proof}\ \\
\begin{enumerate}
	\item The discriminant is equal to
	$$
		\delta(\mathbf{x},\mathbf{y},w)=-d^2\sin^2\alpha.
	$$
	So the discriminant  is zero either if $d=0$, that is, the minimal distance is
	zero and there is a common tangential plane, or $\alpha=0$ and axes
	of the cones are parallel. The discriminant is zero also 
	if cycles in $(\mathbf{x},\mathbf{y})$ are linearly dependent and 
	the families share a common cycle, (see proposition
	\ref{dvasopa1}). 
	\item If the discriminant is negative then the axes are not parallel and 
	each pair of cycles containing a cycle from each family has a tangential
	distance, (see proposition \ref{dvasopa}).
	\item follows from proposition \ref{dvasopa}.
\end{enumerate}
\end{proof}

The meaning of the discriminant of two hyperbolic cones or subcycles can be
expressed in terms of the
special determinants, see section \ref{simplex}.
Let us take a look at cone families. 
The discriminant
$\delta(\mathbf{x},\mathbf{y},w)=\frac{\Delta(\mathbf{X},\mathbf{Y},W)}
{\Delta(\mathbf{X},W)\Delta(\mathbf{Y},W)}$ is
proportional to  the 
squared volume of the simplex spanned by tangential distances of
cycles $\mathbf{x}$ and $\mathbf{y}$ divided by the product of the squared volumes of
the facets spanned by tangential distances of $\mathbf{x}$ and
$\mathbf{y}$. In the case of two lines defined by a pair of points in 
space, the discriminant is equal to the squared volume of the simplex spanned by
all four points divided by the product of squared distances between the points of the
same line.
It follows that 
\begin{proposition}
The discriminant of two hyperbolic cone families is equal zero if the volume of the contact simplex 
spanned by tangential distances is zero, is negative if the volume of the
contact simplex is different from zero and is positive if there is no such
simplex. 
\end{proposition}

In the case of subcycles the discriminant can be expressed in the terms of a quotient of volumes of simplices or polar sines, but the geometrical meaning 
is less  evident.    

Figure \ref{Linked} on the left shows two linked subcycles corresponding to Steiner
pencils $\langle\mathbf{x},r\rangle\cap\Omega$ and 
$\langle\mathbf{y},r\rangle\cap\Omega$. 
The sphere is the minimal sphere of the product family
$\langle\mathbf{x},\mathbf{y},r\rangle\cap\Omega$. 
The circles intersect the sphere in
antipodal points. On the right the figure shows two non intersecting cones. 
Two middle spheres are the spheres of minimal tangential distance. 

\end{document}